\documentclass[12pt]{article}
\usepackage{graphics}
\usepackage{amssymb,amsthm}
\usepackage[titletoc]{appendix}
\usepackage{graphicx}
\usepackage{amsmath}
\usepackage{mathrsfs}
\usepackage{bm}
\usepackage[show]{ed}
\usepackage[colorlinks=false,breaklinks=true,linkcolor=blue]{hyperref}
\usepackage{graphics,psfrag,graphicx,color}

\newcommand{\jump}[1]{\left[\!\left[#1\right]\!\right]}
\newcommand{\ave}[1]{\left\{\!\!\left\{#1\right\}\!\!\right\}}
\newcommand{\triple}[1]{|\!|\!|#1|\!|\!|}

\newtheorem{definition}{Definition}[section]
\newtheorem{proposition}{Proposition}[section]
\newtheorem{corollary}[proposition]{Corollary}
\newtheorem{lemma}[proposition]{Lemma}
\newtheorem{theorem}[proposition]{Theorem}
\numberwithin{equation}{section}
\newcounter{rem}[section]
\renewcommand{\therem}{\thesection.\arabic{rem}}
\newenvironment{rem}{
\refstepcounter{rem} {\vspace{3mm}\noindent\bf Remark \therem :}
\begin{rm}}{\end{rm}\vspace{3mm}}
\title{Time-Domain Boundary Integral Methods in Linear Thermoelasticity}
\author{{\sc G. C. Hsiao}\thanks{Department of Mathematical Sciences, University of Delaware,
Newark, DE 19716-2553, USA \quad
Email: {\tt ghsiao@udel.edu}} \quad
and \; {\sc T. S\'{a}nchez-Vizuet} \thanks{Courant Institute of Mathematical Sciences, New York University, New York, NY 10012-1185 USA\quad
Email: {\tt tonatiuh@cims.nyu.edu}.}}
\date{\small{\it Dedicated to the memory of Francisco-Javier Sayas}}

\begin{document}
\maketitle

\begin{abstract}
This paper is concerned with the application of time domain boundary integral methods(TDBIMs) to a non-stationary boundary value problem for the thermo-elasto-dynamic equations, based on the Lubich approach via the Laplace transform. Fundamental solutions of the transformed thermo-elasto-dynamic equations are constructed explicitly by the H\"{o}rmander method. Simple- and double-layer potential boundary integral operators are introduced in the transformed domain and their coercivity is established.

Based on the estimates of various boundary integral operators in the transformed domain, existence and uniqueness results of solutions are established in the time domain. These results may serve as a mathematical foundation for the semi-discretization and full-discretization schemes based on the boundary element method and convolution quadrature method for time domain boundary integral equations arising in thermoelasticity. 
\end{abstract}

%
\section{Introduction}
%
Boundary integral methods (BIMs) (or rather boundary integral equation methods (BIEMs)) originating from problems of mathematical physics and mechanics have become one of the essential tools for solving boundary value problems for linear elliptic partial differential equations. The numerical discretizations of the variational methods dealing with the weak formulations of boundary integral equations are known as boundary element methods (BEMs). In fact, today, any numerical realization of BIMs is generally referred to as BEMs. They have been proven to be as efficient computationally as finite element methods (FEMs) for constructing numerical solutions of partial differential equations. 

Non-stationary boundary value problems of partial differential equations of parabolic or hyperbolic type can be treated with time domain boundary integral methods (TDBIMs) in the same manner as potential methods for elliptic problems. TDBIMs have received increasing attention in recent years. Roughly speaking, there are three approaches for TDBIMs: one is based on the formulation of problems in a setting of appropriate anisotropic Sobolev spaces, another one based on either the Laplace or Fourier-Laplace transform techniques, and the last one is to treat the time evolution by posing the governing equations as a dynamical system taking values in a Banach Space. 

The former approach dates back to the work of Arnold and Noon \cite{ArNo:1989}, Costabel \cite{Co:1990} and Hsiao and Sarannen \cite{HsSa:1993, HsSa:1991} for the heat equation and Hebeker and Hsiao \cite{HeHs:1993, HeHs:2007} for the Stokes and the Oseen equations in compressible as well as in-compressible fluid flow. For the second approach, we refer to the fundamental work in French school by Bamberger and Ha Duong \cite{BaHa:1986a, BaHa:1986b} for the wave equation, the work by Lubich and Schneider \cite{Lu:1988a, LuSc:1992} for partial differential equations of parabolic and hyperbolic types. The latter approach is more recent and follows the lines laid out in \cite{HaQiSaSa:2015} for the acoustic wave equation and has been applied to the scattering of acoustic waves by elastic and piezoelectric solids in \cite{BrSaSa:2016}.

Lubich's approach has recently been advanced by Sayas and his co-workers (see for instance \cite{LaSa:2009a}, the monograph \cite{Sayas:2016}, and references therein) and has become particularly popular in recent years for treating time-dependent boundary integral equations of convolution type. An essential feature of this approach is that it works directly on the data in the time domain rather in the transformed domain. Moreover, the approach does not depend upon the explicit construction of the time-dependent fundamental solutions of time dependent partial differential equations under considered. 

This paper is concerned with the application of TDBIMs to a non-stationary boundary value problem for the thermo-elasto-dynamic equations, based on the Lubich approach via the Laplace transform. This work, together with \cite{HsSaSaWe:2016}, where the problem was treated through a coupled PDE/BIE formulation, may be considered as a generalization of some results in \cite{HsSaSa:2016} where the purely elasto-acoustic dynamical problem was addressed. It is also related to the work in \cite{BaHaHsSa:2015} for treating the non-stationary Stoke equations, and to \cite{SaSa:2016} where piezoelectric effects were considered. 

The basic governing dynamic equations in thermoelasticity together with their transformed formulations in the Laplace domain---known as the system of thermoelastic pseudo-oscillation operators---are given in Section \ref{sec:GoverningEquations}. This section ends with a solution expansion consisting of solutions to the Yukawa potential equations. The basic thermoelastic oscillation problems are introduced in Section \ref{sec:ThermoelasticOscillation} and Green's representations for the solutions of the partial differential equations are given in terms of the corresponding fundamental solution which are constructed explicitly in the Appendix A by means of H\"{o}rmander's method. In terms of the fundamental solution, four basic thermoelastic potentials are introduced. These will be used to transform the boundary value problems into boundary integral equations by the indirect method based on a layer potential ansatz. We end the section with a theorem summarizing some classical results due to Kupradze \cite[Theorem 3.1 , p. 572 ]{Ku:1979} concerning existence and uniqueness of solutions to thermoelastic pseudo-oscillation problems. 

The last two sections (Section \ref{sec:BoundaryValueProblems} and Section \ref{sec:TimeDomain}) contain the main results of the paper. Section \ref{sec:BoundaryValueProblems} contains the main estimates for the boundary integral operators in terms of the complex parameter $s$ of the Laplace transform. In particular we establish the coercivity of both simple- and double-layer boundary integral operators in the relevant Sobolev spaces. These estimates are crucial for the main results of weak solutions of initial- boundary value problems in the time domain formulated in Section \ref{sec:TimeDomain}, where existence and uniqueness results for two basic model problems are treated. Namely, the initial-Dirichlet and initial-Neumann boundary value problems for the elastic displacement and temperature fields governed by the linear thermoelastic-dynamic equations introduced in Section \ref{sec:GoverningEquations}. These results may serve as the mathematical foundation for semi discrete and fully discrete numerical schemes based on the boundary element method combined with convolution quadrature (see e.g. \cite{Lu:1988a,Lu:1988}) for time-domain boundary integral equations arising in thermoelasticity. 

%
\section{Governing equations}\label{sec:GoverningEquations}
%
Let $\Omega^-$ be a bounded domain in $\mathbb{R}^3$ with Lipschitz boundary $\Gamma$ and let $\Omega^+ := \mathbb{R}^d \setminus \overline\Omega^-$ its exterior. In linear thermoelasticity, the governing equations for the elastic displacement field $\bf {U}$ and temperature variation field $\Theta$ are the thermo-elasto-dynamic equations consisting of an {\it equation of motion}
\begin{align} \label{eq:2.1}
\rho \frac{\partial^2\mathbf{U}} {\partial t^2} - \Delta^{*} \mathbf{U} + \gamma \, \nabla~\Theta 
 = \mathbf{0}, 
 \end{align}
 and an {\it equation of energy}
 \begin{align} \label{eq:2.2}
 \frac{1}{\kappa}\frac{\partial \Theta}{\partial t} - \Delta \Theta + \eta\; \frac{\partial}{\partial t}(\nabla\cdot \mathbf{U}) = 0
 \end{align}
in $\Omega^- \times (0,T)$ (or $\Omega^+ \times ((0, T)$), where $T$ is a given positive constant. Here $ \rho$ is the constant density of the elastic body, $\kappa = k/\delta$ the thermal diffusivity,
$k$ the thermal conductivity of the medium, and $\delta$ the specific heat. The coupling constants $\gamma$ and $\eta$ are defined by
\[
\gamma := (\lambda + 2/3 \,\mu) \alpha, \qquad \text{ and } \qquad \eta := \zeta \Theta_0 /k,
\]
where $\alpha$ is the volumetric coefficient of thermal expansion, and $\Theta_0$ is the reference temperature at the thermoelastic medium 
when it is not subject to external stresses or strains. After a change in the temperature from $\Theta_0$ to $\tilde \Theta$, the increased temperature $\Theta = \tilde \Theta - \Theta_0$ will create a linear displacement field in the medium under that will induce internal stress. 
In the formulation, as usual, $\Delta^*$ is the Lam\'e operator defined by 
\begin{eqnarray*}
 \Delta^* \mathbf{U} &:=& \mu \Delta \mathbf{U} + (\lambda + \mu) \nabla\,\nabla\cdot \mathbf{U}\\
 &=& \nabla\cdot \widetilde{\bf \bm{\sigma}}(\bf U), 
 \end{eqnarray*}
where $\mu$ and $\lambda$ are the Lam\'e constants for the isotropic medium, and $\widetilde{\bm {\sigma}}({\bf U})$ and $\widetilde{\bm \varepsilon} ({\bf U}) $ are the linearized elastic stress and strain tensors which, for an isotropic elastic medium are given respectively by:
 \[
 \widetilde{\bm \sigma} ({\bf U}) = (\lambda\;\nabla\cdot{\bf U}) {\bf I} + 2 \mu \; \widetilde{\bm \varepsilon} ({\bf U}) \quad \mbox{and }\quad
\quad \widetilde{\bm \varepsilon}({\bf U}) = \tfrac{1}{2}( \nabla{\bf U} + (\nabla{\bf U})^\prime),
\]
above, the identity operator is denoted by ${\bf I}$ and the superscript $^\prime$ is used to denote transposition. In a thermoelastic medium, the corresponding stress and strain tensors are
 \[
 {\bm \sigma}({\bf U}, \Theta) = \widetilde{\bm \sigma} ({\bf U}) - \gamma \, \Theta\, {\bf I} \quad \mbox{and }\quad
\quad \bm \varepsilon ({\bf U}, \Theta) = \widetilde{\bm \varepsilon}({\bf U}) - \frac{\gamma}{3 \lambda + 2 \mu} \Theta \,{\bf I}.
\]
These relations, establishing the linear dependence between the strain, temperature and stress, are called Duhamel-Neumann's law. The thermoelastic coupling is measured by the dimensionless constant 
\cite{No:1975} 
\[
\epsilon := \frac{\gamma \eta \kappa}{\lambda + 2 \mu},
\]
which assumes small positive values for most of thermoelastic media. If the thermal effect is neglected, Duhamel-Neumann's law will obviously give Hooke's law of the classical theory for an arbitrary isotropic medium (see, e.g. \cite{Ku:1979}). In the thermoelastic medium, the given physical constants $  \rho, \lambda, \mu, \gamma, \eta,
\kappa$, are assumed to satisfy the inequalities:
\[
 \rho > 0, \;\mu > 0 , \;3 \lambda + 2 \mu > 0, \;\gamma/\eta > 0, \;\kappa > 0.
\]
In order to formulate and analyze the Time Domain Boundary Integral Equations (TDBIE), we will first transform the equations in the Laplace domain. 
Throughout the paper let the complex plane be denoted by $\mathbb{C}$ and its positive  half-plane denoted by 
\[
\mathbb{C}_+:= \{ s \in \mathbb{C} : \mathrm{Re}(s) > 0\}.
\]
We begin with the Laplace transform for an ordinary complex-valued function. Let $F \colon [0, \infty) \to \mathbb{C}$ be a complex-valued function with limited growth at infinity. We denote the Laplace transform of $F$ by 
\[
f(s)= \mathcal{L}{F}(s) := \int_0^\infty e^{-st} F(t) dt,
\]
provided it exists. A common criterion for limited growth at infinity is that $F$ be of exponential order , i.e., that there exists $T>0$ and positive $ M \equiv M(T)$ and $\alpha \equiv \alpha(T)$ such that
\[
|F(t)| \leq M e^{\alpha t} \quad \forall t \geq T.
\]
In the following, to make notation lighter, we will suppress the explicit dependence of the variables with respect to position and the Laplace parameter and will let $\mathbf{u} := \mathbf{u}(x,s)= \mathcal{L}\{ {\bf U}(x,t)\}$ and $ \theta:=\theta(x,s) = \mathcal{L}\{{\Theta(x,t)} \}$. Then, equations \eqref{eq:2.1} and \eqref{eq:2.2} become
\begin{subequations}\label{eq:2.3}
\begin{alignat}{6} 
 \Delta^{*} \mathbf{u} -  \rho s^2 \mathbf{u} - \gamma\, \nabla ~\theta =\,& \mathbf{0} &\quad& \mbox{in}\quad \Omega^-~ (\text{or} ~ \Omega^+), \\
 \Delta \theta - (s/\kappa) ~\theta - s \eta ~\nabla\cdot \mathbf{u} =\,& 0  &\quad& \mbox{in} \quad \Omega^-~(\text{or} ~\Omega^+). 
\end{alignat}
\end{subequations}
The system of equations \eqref{eq:2.3} can be written in terms of the matrix of operators in the form 
\begin{equation} \label{eq:2.5}
{\bf B}(\partial_x , s) \begin{pmatrix}
{\mathbf u} \\
{\theta}\\
\end{pmatrix}
:= \left (
\begin{array} {ll} 
  \Delta^* - \rho s^2 & -\gamma \,\nabla \\
  -s \,\eta\, \nabla^{\top} & \Delta -s/\kappa \\
 \end{array} \right )
 \begin{pmatrix}
{\mathbf u} \\
\theta 
\end{pmatrix}
=  \begin{pmatrix} {\bf 0}\\ 0 \end{pmatrix}\quad \mbox{in} \quad \quad \Omega^- ~(\text{or}~ \Omega^+). 
\end{equation}
The operator ${\bf B}(\partial_x , s)$---referred to as the {\it thermoelastic pseudo-oscillation operator} in \cite{Ku:1979}---is not self-adjoint; its adjoint operator ${\bf B}^*(\partial_x , s)$ may be obtained from ${\bf B}(\partial_x , s)$ by replacing $\gamma$ with $ - s\eta $ and vice versa.

Following Kupradze \cite[p.62]{Ku:1979}, we will say that a vector-valued function $ ({\bf u}, \theta)^{\top} $ in a domain $D$ is \textit{regular}, if $ ({\bf u}, \theta)^{\top} \in C^2(D)^4\cap C^1(\overline{D})^4$ (also denoted by $ {\bf C}^2(D)\cap {\bf C}^1(\overline{D})$ ). As for the time-harmonic thermoelastic-oscillation (Kupradze \cite[p.139 ]{Ku:1979}, see also Cakoni \cite{Cakoni:2000}), the thermoelastic pseudo-oscillation admits a similar solution expansion:
\begin{lemma}
The regular solution $ ({\bf u}, \theta)$ of \eqref{eq:2.5} admits in the domain of regularity a representation 
of the form:
\begin{equation}\label{eq:2.6}
 (\mathbf{u}, \theta) = (\mathbf{u}_1, \theta_1) + (\mathbf{u}_2, \theta_2) + (\mathbf{u}_3, \theta_3), 
 \end{equation}
with $({\bf u}_k, \theta_k), k = 1,2,3 $ satisfying 
\begin{alignat*}{8}
(\Delta -\lambda^2_1)\mathbf{u}_1 =\,& {\bf 0}, \qquad& (\Delta -\lambda^2_2)\mathbf{u}_2 =\,& {\bf 0}, \qquad& (\Delta -\lambda^2_3) \mathbf{u}_3 =\,& {\bf 0}, \\
\nabla \times \mathbf{u}_1 =\,& {\bf 0}, \qquad& \nabla\times\mathbf{u}_2 =\,& {\bf 0}, \qquad& \nabla\cdot \mathbf{u}_3 =\,& 0, \\
(\Delta -\lambda^2_1)\theta_1 =\,& 0, \qquad& (\Delta -\lambda^2_2)\theta_2 =\,& 0, \qquad& \theta_3 =\,& 0.
\end{alignat*} 
\end{lemma}
\begin{proof}
From the first equation of \eqref{eq:2.5}, let
\begin{equation*}
\mathbf {u}^{(1)} :=\, \frac{\lambda + 2 \mu} {\rho s^2} \, \nabla \,\nabla\cdot \mathbf{u} - \frac{\gamma}{\rho s^2} \nabla\, \theta, \quad \text{ and } \quad 
\mathbf {u}^{(2)} :=\, - \frac{\mu} {\rho s^2}\, \nabla\times(\nabla\times\mathbf {u}), 
\end{equation*} 
by making use of the identity $ \Delta \mathbf {u} = \nabla\,\nabla\cdot \mathbf {u} - \nabla\times(\nabla\times\mathbf{u})$. Then $\nabla\times{\bf u}^{(1)} = {\bf 0}$ and $ \nabla\cdot{\bf u}^{(2)} = 0$. From the second equation of \eqref{eq:2.5}, we have the decomposition: 
\[
(\mathbf{u}, \theta) = (\mathbf{u}^{(1)} + \mathbf{u}^{(2)}, \theta)
\]
where $ (\mathbf{u}^{(1)}, \mathbf{u}^{(2)})$ and $\theta$ satisfy the equations:
\begin{alignat*}{8}
(\Delta - \lambda ^2_1) (\Delta - \lambda^2_2) \;\theta =\,& 0, & & \nonumber \\
(\Delta - \lambda^2_1) (\Delta - \lambda^2_2) \;\mathbf{u}^{(1)} =\,& \mathbf{0}, \qquad \qquad&  \nabla \times\mathbf{u}^{(1)} =\,& \mathbf{0}, \\
(\Delta - \lambda^2_3) \;\mathbf{u}^{(2)} =\,& \mathbf{0}, \qquad \qquad& \nabla\cdot\mathbf{u}^{(2)} =\,& \mathbf{0}.
\end{alignat*}
The constants $\lambda^2_1, \lambda^2_2 $ are determined by the equations:
\begin{equation}
\lambda^2_1 + \lambda^2_2  = \frac{s}{\kappa} ( 1 + \epsilon) + \lambda^2_p, \qquad \qquad \lambda^2_1 \lambda^2_2 \;= \left(\frac{s}{\kappa}\right) \lambda^2_p , \label{eq:2.7a} 
\end{equation} 
with $ \lambda^2_3 = \rho s^2/ \mu, \; \lambda^2_p = \rho s^2/ (\lambda + 2 \mu)$ and $\epsilon = \gamma\, \eta\, \kappa / (\lambda + 2 \mu)$. It may be shown that in the domain of regularity, the regular solution of \ref{eq:2.5} is infinitely differentiable. 

Next, the decomposition of $(\mathbf{u}, \theta) = (\mathbf{u}_1 + \mathbf{u}_2 +\mathbf{u}_3, \theta_1 + \theta_2 + \theta_3 ) $ can be obtained according to 
\begin{alignat*}{8}
\mathbf{u}_1  =\,& \frac{(\Delta - \lambda^2_2) } {\lambda^2_1 - \lambda^2_2} \mathbf{u}^{(1)}, \qquad& \mathbf{u}_2 =\,& \frac{(\Delta - \lambda^2_1)} {\lambda^2_2 - \lambda^2_1} \mathbf{u}^{(1)}, \qquad& \mathbf{u}_3 =\,& \mathbf{u}^{(2)}, \\
\theta_1 =\,& \frac{ (\Delta - \lambda^2_2) } {\lambda^2_1 - \lambda^2_2} \theta, \qquad& \theta_2 =\,& \frac{ (\Delta - \lambda^2_1)} {\lambda^2_2 - \lambda^2_1} \theta, \qquad& \theta_3 =\,& 0.
\end{alignat*}
\end{proof}

We remark as in the dynamic elasticity, $ c_s:= \sqrt{\mu}$ and $ c_p:= \sqrt{\lambda + 2 \mu}$
are the phase velocities for the transversal and longitudinal wave, respectively. The 
parameter $\epsilon \in (0,1) $ plays an important role, and for most of bodies, it is much smaller than unity---see, e.g. \cite{No:1975}. When $\epsilon = 0 $, the deformation and temperature fields become separated, and from eq.\eqref{eq:2.7a}, we may obtain, explicitly, the wave numbers 
\[
\lambda^2_1 = \frac{s}{ \kappa}, \qquad \lambda^2_2 =\frac{ \rho\, s^2}{\lambda + 2 \mu}, \qquad \mbox{ and } \qquad \lambda^2_3 = \frac{\rho\, s^2} {\mu}.
\]
%
\section{Thermoelastic pseudo-oscillation problems}\label{sec:ThermoelasticOscillation} 
%
We begin with the first and second Green formulas for the boundary-value problems of thermoelastic pseudo-oscillations in $\Omega^-$. Let $(\mathbf{v},  v)^{\top} $ and $ (\mathbf{u}, \theta)^{\top} \in \mathbf{C}^2(\Omega^-) \cup \mathbf{C}^1(\bar\Omega^-) $ be regular functions in $\Omega^-$. 
Then first Green's formula assumes the form: 
\begin{align} \label{eq:3.1}
\int_{\Omega^-} \mathbf{B} \begin{pmatrix}
{\mathbf u} \\
\theta 
\end{pmatrix} \cdot \begin{pmatrix}
{\mathbf v} \\
 v 
\end{pmatrix} 
dx &+ \mathcal{A}_{\Omega^-} \Big( \begin{pmatrix}
{\mathbf u} \\
\theta 
\end{pmatrix}, \begin{pmatrix}
{\mathbf v} \\
 v
\end{pmatrix} \Big) = \int_{\Gamma}\mathcal{R}_N \begin{pmatrix}
{\mathbf u} \\
\theta 
\end{pmatrix} \cdot \begin{pmatrix}
{\mathbf v} \\
 v 
\end{pmatrix}
d\Gamma, 
\end{align}
while the second Green formula is given by 
\begin{align}
\int_{\Omega^-}\mathbf{B} \begin{pmatrix}
{\mathbf u} \\
\theta 
\end{pmatrix} \cdot \begin{pmatrix}
{\mathbf v} \\
 v 
\end{pmatrix} 
dx & - \int_{\Omega^-} \begin{pmatrix}
{\mathbf u} \\
\theta 
\end{pmatrix} \cdot \mathbf{B}^* \begin{pmatrix}
{\mathbf v} \\
v 
\end{pmatrix} dx \nonumber \\ 
&= \int_{\Gamma} \mathcal{R}_N \begin{pmatrix}
{\mathbf u} \\
\theta 
\end{pmatrix} \cdot \begin{pmatrix}
{\mathbf v} \\
v 
\end{pmatrix}
d\Gamma - \int_{\Gamma} \begin{pmatrix}
{\mathbf u} \\
\theta 
\end{pmatrix}\cdot \mathcal{R}^{*}_N \begin{pmatrix}
{\mathbf v} \\
 v 
\end{pmatrix} d\Gamma. \label{eq:3.2} 
\end{align}
In the first Green's formulae, $\mathcal{A}_{\Omega^-}$ is a bilinear (or sesquilinear in the complex case) form defined by 
\begin{align*} 
 \mathcal{A}_{\Omega^-} \left( \begin{pmatrix}
{\mathbf u} \\
\theta 
\end{pmatrix}, \begin{pmatrix}
{\mathbf v} \\
 v
\end{pmatrix} \right)&:= \int_{\Omega^-}\Big( \bm{\sigma}( \mathbf{u} ) : \bm{ \varepsilon }( \mathbf{ v}) + \rho s^2 \mathbf{u} \cdot \mathbf{v} - \gamma~\theta\nabla\cdot\mathbf{v} \nonumber  \\
&  \quad \qquad + \,s\eta \nabla\cdot\mathbf{u}\,  v + \nabla \theta \cdot \nabla  v + \frac{s}{\kappa} \theta\,  v \Big)dx, 
\end{align*}
$\mathbf{B}^*(\partial_x, s)$ is the adjoint operator to $\mathbf{B}(\partial_x, s)$ defined by
\[
\mathbf{B}^*(\partial_x , s) \begin{pmatrix}
{\mathbf v} \\
{ v}\\
\end{pmatrix}
:= \left (
\begin{array} {ll}
  \Delta^* - \rho s^2 & s\,\eta \;\nabla \\
  \gamma~ \; \nabla^{\top} & \Delta - s/\kappa \\
 \end{array} \right )
 \begin{pmatrix}
{\mathbf v} \\
 v 
\end{pmatrix},
\] 
and $\mathcal{R}_N$ and $\mathcal{R}^*_N$ are boundary operators defined by 
\begin{align*} 
\mathcal{R}_N(\partial_x , s) \begin{pmatrix}
{\mathbf u} \\
{\theta}\\
\end{pmatrix}\Big |_{\Gamma} 
&:= \left (
\begin{array} {ll} 
 \mathbf{T} & -\gamma\, \mathbf{n} \\ 
  0 & ~~\partial_n \\
 \end{array} \right )
  \begin{pmatrix}
{\mathbf u} \\
\theta 
\end{pmatrix}\Big |_{\, \Gamma} \\
\mathcal{R}^*_N(\partial_x , s) \begin{pmatrix}
{\mathbf v} \\
{ v}\\
\end{pmatrix}\Big |_{\Gamma} 
&:=\left (
\begin{array} {ll} 
  \mathbf{T} & s\, \eta\, \mathbf{n} \\ 
  0 & ~~ \partial_n\\
 \end{array} \right )
 \begin{pmatrix}
{\mathbf v} \\
 v 
\end{pmatrix}\Big |_{\Gamma}
\end{align*} 
where $\mathbf{T}$ is the elastic normal traction operator defined by
\[
T \mathbf{u}_{|\, \Gamma} : = \widetilde{\bm \sigma} (\mathbf{u})_{ |\, {\Gamma}} \;\mathbf{n} = \Big( \lambda (\nabla\cdot\mathbf{u} )\mathbf{n} + 2 \mu \frac{\partial \mathbf{u}} {\partial n} + \mu \mathbf{n} \times \nabla \mathbf{u}\Big)\Big|_{\, \Gamma}. 
\]
From the first Green formula \eqref{eq:3.1}, it follows that by restricting the test pair $(\mathbf{v},  v)^{\top} $ to various subspaces, we may formulate four basic boundary value problems for the 
partial differential equation \eqref{eq:2.5} in $\Omega^-$ (or $ \Omega^+$) according to the following prescribed boundary conditions, namely, 

\begin{description}
\item{\bf I.} The Dirichlet-type boundary condition: 
\begin{align*} 
\mathcal{R}_D \begin{pmatrix}
{\mathbf u} \\
{\theta}\\
\end{pmatrix}\Big |_{\Gamma} 
&:= \left (
\begin{array} {ll}
 \mathbf{I}_3 & 0\\
  0 & 1\\
 \end{array} \right )
 \begin{pmatrix}
{\mathbf u} \\
\theta
\end{pmatrix} \Big |_{\Gamma} = \begin{pmatrix}
{\mathbf f} \\
f 
\end{pmatrix} \quad \mbox{on} \quad \Gamma. 
\end{align*}
\item {{\bf II.}} The Neumann-type boundary condition: 
\begin{align*} 
\mathcal{R}_N \begin{pmatrix}
{\mathbf u} \\
{\theta}\\
\end{pmatrix}\Big |_{\Gamma} 
&:= \left (
\begin{array} {ll}
 \mathbf{T} & -\gamma\, \mathbf{n} \\ 
  0 & ~~\partial_n \\
 \end{array} \right )
  \begin{pmatrix}
{\mathbf u} \\
\theta 
\end{pmatrix}\Big |_{\Gamma} 
= \begin{pmatrix} 
 \mathbf{g} \\
g
\end{pmatrix} \quad \mbox{on} \quad \Gamma.
\end{align*}
\item{{\bf III.}} The Dirichlet and Neumann mixed-type boundary condition:
\begin{align*}
\mathcal{R}_{DN} \begin{pmatrix}
{\mathbf u} \\
{\theta}\\
\end{pmatrix}\Big |_{\Gamma} 
&:= \left (
\begin{array} {ll} 
 \mathbf{I}_3 & 0 \\
  0 & \partial_n \\
 \end{array} \right )
 \begin{pmatrix}
{\mathbf u} \\
\theta 
\end{pmatrix}\Big |_{\Gamma} = \begin{pmatrix} \mathbf{f}\\
g
\end{pmatrix}\quad \mbox{on}\quad \Gamma. 
\end{align*}
\item{{\bf IV.}} The Neumann and Dirichlet mixed-type boundary condition:
\begin{align}
 \mathcal{R}_{N D} \begin{pmatrix}
{\mathbf u} \\
{\theta}\\
\end{pmatrix}\Big |_{\Gamma} 
&:=\left (
\begin{array} {ll} 
 \mathbf{T}& -\gamma\, \mathbf{n} \\ 
  0 & ~~ -1 \\
 \end{array} \right )
  \begin{pmatrix}
{\mathbf u} \\
\theta 
\end{pmatrix}\Big |_{\Gamma} 
 = \begin{pmatrix}
 \mathbf{g} \\
- f 
\end{pmatrix}\quad \mbox{on} \quad\Gamma. \label{eq:3.11}
\end{align}
\end{description}

Above, $\mathbf{f}, f $ and $ \mathbf{g}, g $ are given Dirichlet and Neumann boundary data satisfying certain regularity conditions to be specified. In the definition for the boundary condition for $\theta$ in {\bf IV}, we have tacitly employed the negative Dirichlet condition instead (see Kupradze \cite[p.538]{Ku:1979} for an obvious reason to be clear).

The second Green formula \eqref{eq:3.2} gives a representation for the solution of eq.\eqref{eq:2.5}, if $({\bf v}, v)$ is replaced by the fundamental solution of the adjoint partial differential operator $\mathbf{B}^*(\partial_y, s)$. In Appendix A, we have derived the fundamental solution $ \underline {\underline{\bf E}}(x,y;s)$ for the operator $\mathbf{B}(\partial_x, s)$ based on H\"{o}rmander's approach \cite{Ho:1964} (see also Kupradze \cite{Ku:1979}). In terms of $ \underline {\underline{\bf E}}(x,y;s)$, we have the representation of $({\mathbf u}, \theta)^{\top}$ in the form: 
\begin{align}
 \begin{pmatrix}
{\mathbf u} \\
\theta 
\end{pmatrix} (x)
& = \int_{\Gamma} \underline {\underline{\bf E}}(x,y;s)\, {\mathcal R}_N \begin{pmatrix}
{\mathbf u} \\
\theta 
 \end{pmatrix} d_y\Gamma - \int_{\Gamma} \Big( R^*_y ~\underline {\underline{\bf E}}^{\top}(x,y;s) \Big)^{\top} \begin{pmatrix}
{\mathbf u} \\
\theta 
\end{pmatrix}
d_y\Gamma , \quad x \in \Omega^-, \label{eq:3.12} 
\end{align}
which will lead to the direct method for reducing the Dirichlet and the Neumann problems (i.e., Problem (I) and Problem (II)) to boundary integral equations on $\Gamma$. In this cases, the solutions of the integral equations are explicitly expressed in terms of solutions of the boundary value problems (i.e. their restrictions to the boundary and their outer normal derivatives). However, for the mixed type of the Dirichlet and the Neumann problems (Problem (III) and Problem (IV)), the direct representations leading to the boundary integral equations are not so straightforward. On the other hand, the indirect method based on the layer Ansatz can be easily applied to all of the the four boundary value problems without any difficulty. For this purpose, we introduce the four basic thermoelastic-layer potentials. We begin with the familiar simple- (or single-) and double-layer potentials from (\ref{eq:3.12}):
\begin{align*}
 \mathcal{S} (s)~\begin{pmatrix} {\bm{\lambda}}\\
\varsigma
\end{pmatrix} (x) &:= \int_{\Gamma} \underline {\underline{\bf E} } (x,y;s)~\begin{pmatrix} 
\bm{\lambda}\\
\varsigma
\end{pmatrix} (y)\, d_y{\Gamma}\quad x \not \in \Gamma,\\
 \mathcal{D} (s)~\begin{pmatrix}{\bm{\phi}}\\
\varphi
\end{pmatrix} (x) &:= \int_{\Gamma} \Big( \mathcal{R}^{*}_{N_y} ~\underline {\underline{\bf E}}^{\top}(x,y;s)\Big)^{\top}
\begin{pmatrix} \bm{\phi} \\
\varphi 
\end{pmatrix} (y) \,d_y{\Gamma} \quad x \not \in \Gamma. 
\end{align*}
 We also need two additional thermoelastic-layer potentials 
\begin{align}
 \mathcal{Q}_{DN}(s)~\begin{pmatrix} {\bm{\lambda}}\\
\varphi
\end{pmatrix} (x) &:= \int_{\Gamma} \Big( \mathcal{R}^{*}_{DN_y} \underline {\underline{\bf E}}^{\top}(x,y;s) \Big)^{\top}
\begin{pmatrix} 
\bm{\lambda}\\
\varphi
\end{pmatrix} (y)\, d_y{\Gamma} \, \quad x \not \in \Gamma, \label{eq:3.15}\\
 \mathcal{Q}_{ND}(s)~\begin{pmatrix} {\bm{\phi}}\\
\varsigma
\end{pmatrix} (x) &:= \int_{\Gamma} \Big( \mathcal{R}^{*}_{ND_y} \underline {\underline{\bf E}}^{\top}(x,y;s) \Big)^{\top}
\begin{pmatrix} 
\bm{\phi}\\
\varsigma
\end{pmatrix} (y)\, d_y{\Gamma} \, \quad x \not \in \Gamma. \label{eq:3.16}
\end{align}
In the above definitions, ${\bm \lambda}, \varphi , {\bm \phi} $ and $\varsigma$ are unknown density functions satisfying certain regularity properties to be specified. Here the adjoint operators $ \mathcal{R}^{*}_{DN_y} $ and $ \mathcal{R}^{*}_{ND_y} $ are, respectively, defined by 
\begin{align*}
\mathcal{R}^{*}_{DN_y} = \mathcal{R}_{DN_y} &:= \left (
\begin{array} {ll} 
 \mathbf{I}_3 & 0 \\ 
  0 & \partial_n \\
 \end{array} \right ), \quad \mathcal{R}^{*}_{N D_y} :=\left (
\begin{array} {ll}
 \mathbf{T}& \; s\, \eta \, \mathbf{n} \\  
  0 & ~~ -1 \\
 \end{array} \right ).
 \end{align*} 
We note that the simple-layer potential can be rewritten in the from: 
\begin{align*}
 \mathcal{S} (s)~\begin{pmatrix} {\bm{\lambda}}\\
\varsigma
\end{pmatrix} (x) &:= \int_{\Gamma} \Big( \mathcal{R}^{*}_{D_y}\underline {\underline{\bf E} }^{\top} (x,y;s) \Big)^{\top}~\begin{pmatrix} 
\bm{\lambda}\\
\varsigma
\end{pmatrix} (y)\, d_y{\Gamma}, \, \quad x \not \in \Gamma, 
\end{align*}
since $\mathcal{R}^{*}_{D_y} = \mathcal{R}_{D} = \mathbf{I}_4 $. This clearly suggests that the kernels of the mixed potential defined in \eqref{eq:3.15} and in \eqref{eq:3.16} should be a combination of the kernels of the simple- and double potentials. We have deliberately denoted the density functions for the mixed potentials according to a corresponding combination of $\bm \lambda, \varsigma $ and $\bm \phi, \varphi $. For the reduction of the four basic boundary value problems for equation \eqref{eq:2.5} to boundary integral equations, simple- and double- layer potential are employed for the problems with boundary conditions {\bf I} and {\bf II}, while the mixed-typed potentials will be utilized for the problems with boundary conditions {\bf III} and {\bf IV}. 
 
To apply the TDBIMs to initial boundary value problems for the dynamic equations in thermoelasicity, we need the weak solutions of the four boundary value problems in terms of density functions of the corresponding potentials for {\eqref{eq:2.5}. In this paper, we will confine ourselves only to Dirichlet and Neumann problems via simple- and double- layer potentials $ \mathcal{S} (s)$ and $ \mathcal{D}(s) $. For the mixed Dirichlet and Neumann - type problems via the mixed type potentials $ \mathcal{Q}_{DN}(s)$ and $ \mathcal{Q}_{ND} (s) $, we will pursue our investigations in a separate communication. 

In the following we collect some basic properties of the four aforementioned thermo-elastic potentials as well as solutions of the interior and exterior four boundary value problems represented in terms of these potentials. To simplify the presentation, we need to introduce some further notation: If the function $\bm \varphi$, defined in $\Omega^- \;( \text{or}\; \Omega^+:= \mathbb{R}^3 \setminus \overline{\Omega^-} ) $, is continuously extendible at a point $y \in \Gamma$, then $\bm \varphi^-(y)$ and $\bm\varphi^+(y)$ will denote the limits
\[
\bm\varphi^-(y) \equiv \left\{ \bm\varphi(y)\right\}^-:= \lim_{ \Omega^- \ni x \to y \in \Gamma}\, \bm\varphi(x), \qquad \mbox{and} \qquad \bm\varphi^+(y) \equiv \left\{ \bm\varphi(y)\right\}^+:= \lim_{ \Omega^+ \ni x \to y \in \Gamma}\, \bm\varphi(x).
\]
We also adopt the following conventions: for any function $v$ defined on both sides of $\Gamma$ its jump across the boundary $\Gamma$ will be denoted by $\jump{v} := ( v ^- -  v^+) $, while the average of its traces from inside and outside of $\Gamma$ will be denoted by $ \ave{v } := ( v^- + v^+ ) /2$. For vector valued functions, the jump and average are defined component wise in a similar fashion.

We now summarize some of the classical available results according to [3.1 Theorem, p.572 ]\cite{Ku:1979} in the following theorem. We will use familiar notation  from potential theory (see e.g. the monograph \cite{HsWe:2008}), to refer to the boundary integral operators. 
\begin{theorem} If $\mathrm{Re}(s) \geq \sigma_0 > 0$, then interior and exterior BVPs for equation \eqref{eq:2.5} with boundary conditions {\bf I}, {\bf II}, {\bf III}, {\bf IV} are solvable. The solutions are unique and can be represented by the appropriate thermoelastic potentials. The density functions of the integral representation will then be the unique solutions of Boundary Integral Equations (BIE) involving Boundary Integral Operators (BIO) with Cauchy singular kernel. More precisely, we have the formulations 
\begin{align*}
\intertext{ For Problem {\em(I)}, }
\mbox{ \em Rep : }&({\bf u}, \theta)^{\top} =  \mathcal{D} (s) ({\bm \phi}, \varphi)^{\top}\; \mbox{in} \; \;\Omega^-\; (or\; \Omega^+),\quad \left(\mathcal{R}_D({\bf u}, \theta)^{\top}\right)^{\mp} = ({\bf f},f)^{\top} \quad \mbox{on} \quad \Gamma,\\
\mbox{\em BIE : }&\Big( \mp \tfrac{1}{2}\mathcal{ I} + \mathcal{K}(s) \Big) ({\bm \phi}, \varphi )^{\top} = ({\bf f}, f) ^{\top} \quad \mbox{on}\quad \Gamma, \\
\mbox{\em BIO : }&\mathcal{K}(s) ({\bm \phi}, \varphi)^{\top} := \ave{\int_{\Gamma} \Big(
\mathcal{R}^{*}_{N_y}\underline {\underline{\bf E} }^{\top} (x,y;s)\Big)^{\top} ({\bm \phi}, \varphi)^{\top}d\Gamma_y}; 
\intertext{ For Problem {\em (II)}, } 
\mbox{ \em Rep : }&({\bf u}, \theta)^{\top} =  \mathcal{S} (s) ({\bm \lambda}, \varsigma)^{\top}\; \mbox{in} \; \;\Omega^-\; (or\; \Omega^+),\quad \left(\mathcal{R}_N ({\bf u}, \theta)^{\top}\right)^{\mp} = ({\bf g} , g)^{\top}\; \mbox{on}\quad \Gamma,\\
\hspace{-8mm}\mbox{\em BIE : }&\Big( \pm \tfrac{1}{2}\mathcal{ I} + \mathcal{K}^\prime(s) \Big) ({\bm \lambda}, \varsigma )^{\top} = ({\bf g}, g) ^{\top} \;\;\mbox{on}\quad \Gamma, \\
\hspace{-8mm}\mbox{\em BIO : }&\mathcal{K}^{\prime}(s)({\bm \lambda}, \varsigma)^{\top} := \ave{ \int_{\Gamma} \Big(
\mathcal{R}_{N_x}\underline {\underline{\bf E} }(x,y;s)\Big)({\bm \lambda}, \varsigma)^{\top}d\Gamma_y} ;
\intertext{For Problem {\em (III)}, }
\mbox{ \em Rep : }&({\bf u}, \theta)^{\top} =  \mathcal{Q}_{ND} (s) ({\bm \phi}, \varsigma)^{\top}\; \mbox{in} \; \;\Omega^-\; (or\; \Omega^+),\quad \left(\mathcal{R}_{DN}({\bf u}, \theta)^{\top}\right)^{\mp} = ({\bf f} , g)^{\top}\; \mbox{on}\quad \Gamma,\\
\hspace{-4mm}\mbox{\em BIE : }&\Big( \mp \tfrac{1}{2}\mathcal{ I} + \mathcal{K}_{DN}(s) \Big) ({\bm \phi}, \varsigma )^{\top} = ({\bf f}, g) ^{\top} \;\;\mbox{on}\quad \Gamma, \\
\hspace{-4mm}\hspace{-4mm}\mbox{\em BIO : }&\mathcal{K}_{DN} (s) ({\bm \phi}, \varsigma)^{\top} := \ave{ \int_{\Gamma} 
\mathcal{R}_{DN_x} \Big( {\mathcal R}^{*}_{ND_y} \underline {\underline{\bf E} }^{\top} (x,y;s) \Big)^{\top}
({\bm \phi}, \varsigma)^{\top} d\Gamma_y} ;
\intertext{For Problem {\em (IV)}, }
\mbox{ \em Rep : }&({\bf u}, \theta)^{\top} =  \mathcal{Q}_{DN} (s) ({\bm \lambda}, \varphi)^{\top}\; \mbox{in} \; \;\Omega^-\; (or\; \Omega^+),\quad \left(\mathcal{R}_{ND}({\bf u}, \theta)^{\top}\right)^{\mp} = ({\bf g} , -f)^{\top}\; \mbox{on}\quad \Gamma\\
\mbox{\em BIE : }&\Big( \pm \tfrac{1}{2}\mathcal{ I} + \mathcal{K}_{ND}(s) \Big) ({\bm \lambda}, \varphi )^{\top} = ({\bf g},- f) ^{\top} \;\;\mbox{on}\quad \Gamma \\
\mbox{\em BIO : }&\mathcal{K}_{ND} (s) ({\bm \lambda}, \varphi)^{\top} := \ave{\int_{\Gamma} 
\mathcal{R}_{ND_x} \Big( {\mathcal R}^{*}_{DN_y} \underline {\underline{\bf E} }^{\top} (x,y;s) \Big)^{\top}
({\bm \lambda}, \varphi)^{\top} d\Gamma_y} ;
\end{align*}
In the above formulations, all the given boundary data $ {\bf f}, {\bf g}, f $ and $g $ 
as well as the unknown density functions ${\bm \phi}, {\bm \lambda}, \varphi $ and $ \varsigma$ are assumed to be in the classical H\H{o}lder function spaces ${\bf C}^{\alpha}(\Gamma)$ or $ C^{\alpha}(\Gamma)\;(0 < \alpha < 1) $, while the boundary $\Gamma$ is required to be of class $C^{1 + \alpha} $.
\end{theorem}

\begin{rem} A few direct observations and constructive comments on the above theorem should be in order. (a) Problems (III) and (IV) show that it is necessary to define the boundary operator $\mathcal{R}_{ND}$ in eq. \eqref{eq:3.11} with -1, not +1, for the coefficient of $\theta$. Otherwise, the jumps of both density functions $(\bm{\phi}, \varsigma) $ from the boundary integral equation in Problem (III) will not be the same. The same conclusion holds also for the Problem (IV). (b) It is worth mentioning that in contrast to time-harmonic thermoelastic (and elastic) oscillations, no radiation conditions are required for exterior boundary value problems in the thermoelastic pseudo-oscillations. This perhaps can be explained from the expansion of a regular solution of \eqref{eq:2.5}, which is in terms of solutions of Yukawa potential equations, while the former expansions are in terms of solutions of Helmholtz equations. (c) All of four boundary value problems are reduced to boundary integral equations of the second kind from which the density functions can be computed via jumps of solutions of the corresponding boundary values : Problems (I): $\jump{\mathcal{R}_D({\bf u}, \theta)^{\top}} = - (\bm{\phi}, \varphi)^{\top}$; Problem (II): $\jump{\mathcal{R}_N(\mathbf{u}, \theta)^{\top}} = +(\bm{\lambda}, \varsigma)^{\top} $; Problem(III): $ \jump{\mathcal{R}_{DN}(\mathbf{u}, \theta)^{\top}} = - ( {\bm \phi}, \varsigma)^{\top} $; Problem(IV): $ \jump{\mathcal{R}_{ND}(\mathbf{u}, \theta)^{\top}} = + ( {\bm \lambda}, \varphi)^{\top} $. (d) The main drawback of Kupradze's approach is that it concerns only with the traditional second kind boundary integral equation formulations in H\"{o}lder function spaces. This will exclude many important practical applications. 
\end{rem} 

In the next section, we will re-examine Problem (I) and Problem (II) (the Dirichlet and the Neumann problems) in Sobolev spaces. We will show how to apply the method of boundary integral equations of the first kind to both problems. In particular we need sharp estimates on the s-parameter dependence of solutions in the Laplace transformed domain. This information is crucial in order to apply TDBIMs to the problems in Section \ref{sec:TimeDomain}. It is not difficulty to see that mixed Dirichlet and Neumann Problem (III) and problem(IV) can be treated in the same manner and we will report our investigation for the mixed problems 
in a separate communication. 
%
\section{Dirichlet and Neumann Boundary Value Problems} \label{sec:BoundaryValueProblems}
%
We begin with the definitions of of the following energy norms in the domain $D= \Omega^-$ or $\Omega^+$: 
\begin{alignat*}{6}
\triple{\mathbf u}_{|s|, D}^2 &:= \left( \widetilde{\bm{ \sigma}} ( {\mathbf u}), \overline{\widetilde{\bm{\varepsilon}} (\mathbf {u}} ) \right)_{D} +  \rho \| s \; \mathbf u \|^2_{D} \quad& \mathbf u &\in {\mathbf H}^1(D), \\
 \triple{\theta}^2_{|s|, D} &:= \| \nabla \theta\|^2_{D} + \kappa^{-1}
 \| \sqrt{ |s|} \; \theta \|_{D}^2 \quad& \theta &\in H^1(D), \\
 \triple{ ({\bf u}, \theta)}^2_{|s|, D} &:= \triple{\mathbf u}^2_{|s|, D} + \triple{\theta}^2_{|s|, D}
 \quad & ({\bf u}, \theta) &\in {\mathbf H}^1(D)\times H^1(D). 
 \end{alignat*}
(see, e.g. \cite{HsSaSa:2016, HsSaSaWe:2016}). In the sequel, it will be convenient to use the following notation for the real part of the Laplace parameter
\[
\sigma : = \mathrm{Re}(s) \qquad\text{ and }\qquad \underline{\sigma}:= \min\{1, \sigma\}.
\]
We also need the following result that establishes the equivalence between the above energy norms and the usual Sobolev norm in $D$. 
 \begin{align}
\underline{\sigma }\triple{\mathbf{u}}_{1, D} &\leq \triple{\mathbf{u}}_{|s|, D} \leq
\frac{|s|}{\underline{\sigma}} \triple{\mathbf{u}}_{1, D} ,\nonumber\\
\sqrt{\underline{\sigma}} \triple{\theta}_{1, D} &\leq \triple{\theta}_{|s|, D} \leq 
\sqrt{\frac{|s|}{\underline {\sigma} } } \triple{\theta}_{1,D}, \nonumber \\
\underline{\sigma} \triple{({\bf u}, \theta)}_{1, D} &\leq \triple{({\bf u}, \theta)}_{|s|, D} \leq \frac{|s|}{{\underline{\sigma}}^{3/2}}\triple{({\bf u}, \theta)}_{1, D} , \label{eq:4.6}
 \end{align}
which can be obtained from the inequalities: 
 \[
\underline{\sigma} \leq \min\{1, |s|\},\quad \mbox{and} \quad \underline{\sigma}\,\max\{1, |s|\}
 \leq |s|,~ \;\forall s \in \mathbb{C}_+.
 \]
We remark that the norm $\triple{\theta}_{1, \Omega^-} $ is equivalent to $\|\theta\|_{H^1(\Omega^-)} $ and so is the energy norm $\triple{ \mathbf{u}}_{1, \Omega^-}$ equivalent to the $\mathbf{H}^1(\Omega^-)$-norm of ${ \mathbf{u}} $ by the second Korn inequality \cite{Fi:1972}. In what follows, constants ${c_j}'s $ are generic constants independent of $s$ which may or not not be the same at different places. 
 
For a Lipschitz boundary $\Gamma \in C^{0,1} $, lets recall the definition of simple- (or single-) 
and double-layer potentials: 
\begin{align*}
 \mathcal{S} (s) \, (\bm{\lambda},\, \varsigma)& : 
 =\int_{\Gamma} \underline {\underline{ \bf E}}(x,y;s)(\bm{\lambda}, \varsigma)\, d{\Gamma}_y, \quad x \not \in \Gamma \\
 \mathcal{D} (s) \, ({\bm{\phi}}, \, \varphi)&: 
=\int_{\Gamma} \Big( R^*_{N_y} ~\underline {\underline{\bf E}}^{\top}(x,y;s) \Big)^{\top}(\bm{\phi},
\varphi ) \, d{\Gamma}_y , \quad x \not \in \Gamma ,
\end{align*} 
for $(\bm{\lambda}, \varsigma)\in {\bf H}^{-1/2}(\Gamma) \times H^{-1/2} (\Gamma)$ and $ ({\bm \phi}, \varphi) \in {\bf H}^{1/2}(\Gamma) \times H^{1/2} (\Gamma)$. Here and in the sequel, for simplicity, we suppress the symbol $\top$ and write $({\bm \lambda}, \varsigma) , ({\bm \phi}, \varphi) $ instead of $( {\bm \lambda}, \varsigma)^{\top}, ({\bm \phi}, \varphi)^{\top} $, etc. Now, following the standard procedure in potential theory associated to the potentials (see, e.g. Hsiao and Wendland \cite [Chapter 5, 8]{HsWe:2008})  we consider the following transmission problem for the simple-layer potential $  \mathcal{S} (s) $.

\begin{proposition} \label{pro:4.1}
Let $ ({\bm{\lambda}}, \varsigma) \in {\bf H}^{-1/2}(\Gamma) \times H^{-1/2}(\Gamma) $ and consider the function $({\bf u}_{\lambda}, \theta_{\varsigma}) :=  \mathcal{S} (s)(\boldsymbol{\lambda}, 
 \varsigma) \in {\bf H}^1(\mathbb{R}^3\setminus \Gamma) \times H^1(\mathbb{R}^3\setminus \Gamma)$. Then 
\begin{subequations} \label{eq:4.9}
\begin{align}
\mathbf{B}(\partial_x , s) (\mathbf u_{\lambda}, {\theta}_{\varsigma}) 
& = {\bf 0} \quad \mbox{in}\quad \mathbb{R}^3 \setminus \Gamma, \label{eq:4.9a}\\
\jump{ ({\bf u}_{\lambda}, \theta_{\varsigma}) } & = {\bf 0},\label{eq:4.9b}\\
\jump{ \mathcal{R}_N ({ \bf u}_{\lambda}, \theta_{\varsigma})} & = (\boldsymbol{\lambda}, \varsigma). \label{eq:4.9c}
\end{align}
\end{subequations}
Moreover, a pair $({\bf u}_{\lambda}, \theta_{\varsigma}) \in {\bf H}^1(\mathbb{R}^3\setminus \Gamma) \times H^1(\mathbb{R}^3\setminus \Gamma) $ is a solution of {\em(}\ref{eq:4.9}{\em )} if and only if 
$ ({\mathbf u}_{\lambda}, \theta_{\varsigma})$ are such that for all $( {\bf v},  v) \in {\bf H}^1(\mathbb{R}^3\setminus \Gamma) \times H^1(\mathbb{R}^3\setminus \Gamma)$ with $\jump{({\bf v},  v)} ={\bf 0}$, 
\begin{equation}\label{eq:4.10} 
\mathcal{A} _{ \mathbb{R}^3 } \left( ( {\bf u}_{\lambda}, \theta_{\varsigma} ), ( {\bf v},  v) \right) =\left\langle \jump{\mathcal{R}_N ({ \bf u}_{\lambda}, \theta_{\varsigma})}, \overline{({\bf v},  v)}|_{\Gamma} \right\rangle_{\Gamma}=
\left\langle ( \boldsymbol{\lambda}, \varsigma ), \overline{({\bf v},  v)}|_{\Gamma} \right\rangle_{\Gamma},
\end{equation}
where $({\bf v},  v)|_{\Gamma}= \left\{({\bf v},  v)\right\}^{\pm}.$
\end{proposition}
\begin{proof}
Regularity is guaranteed by the properties of the integral formulations of the layer potentials. The differential equation \eqref{eq:4.9a} is satisfied pointwise by differentiating directly the fundamental solution, and therefore it is satisfied in a distributional sense. Condition \eqref{eq:4.9b} is a direct consequence of the fact that $({\bf u}_{\lambda}, \theta_{\varsigma}) \in {\bf H}^1(\mathbb{R}^3 \setminus \Gamma) \times H^1(\mathbb{R}^3 \setminus \Gamma)$. Finally, 
condition \eqref{eq:4.9c} follows from \eqref{eq:4.9a} and \eqref{eq:3.1}. The equivalence of \eqref{eq:4.9} and \eqref{eq:4.10} is straightforward.
\end{proof}
We now introduce the thermoelastic simple-layer (or single layer) operator:
\begin{definition} The thermoelastic simple-layer (or single-layer) operator $ \mathcal{V}(s)$ is a continuous linear mapping ${\bf H}^{-1/2}(\Gamma) \times H^{-1/2}(\Gamma) \rightarrow \ {\bf H}^{1/2}(\Gamma) \times H^{1/2}(\Gamma) $ defined by
\[
\mathcal{V}(s)({\bm \lambda}, \varsigma) := \ave{\mathcal {R}_{D}  \mathcal{S}(s)({\bm \lambda}, \varsigma)}\quad \mbox{for} \quad ({\bm \lambda}, \varsigma) \in {\bf H}^{-1/2}(\Gamma) \times H^{-1/2}(\Gamma).
\]
\end{definition}
Consider the re-scaling 
\[
Z(s) := \left( \begin{array} {ll}
 \overline{s}\,\mathbf{I}_3 & 0 \\
  0 & \gamma/\eta
 \end{array} \right ),
\]
and note the following bound that will be useful in the following arguments
\[
\|Z(s)\|_{\max} \lesssim \frac{|s|}{\underline{\sigma}}.
\]
A simple computation shows that the variational equation \eqref{eq:4.10} is equivalent to
\begin{align}\label{eq:4.13}
 \mathcal{A}_{\mathbb{R}^3\setminus \Gamma}\left( Z(s)  ({\bf u}_{\lambda}, \theta_{\varsigma} ), \overline {( {\bf v},  v)} \right)& = \left\langle Z(s)(\boldsymbol{\lambda},\varsigma ), \overline{({\bf v},  v)}|_{\Gamma} \right\rangle_{\Gamma}. 
\end{align}
 Hence 
 \begin{align}
 \mathrm{Re} \left( \Big\langle Z(s) (\boldsymbol{\lambda},\varsigma), \overline{ {\mathcal{V}}(s) (\boldsymbol{\lambda},\varsigma) }\Big\rangle_{\Gamma} \right) & =
\mathrm{Re}\left(\mathcal{A}_{\mathbb{R}^3\setminus \Gamma} \left(  Z(s)( {\bf u}_{\lambda}, \theta_{\varsigma} ), \overline{ ( {\bf u_{\lambda}}, \theta_{\varsigma})} \right)\right) \nonumber \\
&\geq \underline{\sigma}\triple{\mathbf u_{\lambda}}^2_{|s|,\mathbb{R}^3\setminus \Gamma} +\frac{\gamma}{\eta}\frac{\sigma}{|s|}\triple{\theta_{\varsigma}}^2_{|s|, \mathbb{R}^3\setminus \Gamma} \nonumber \\
&\geq ~ c \,\frac{\sigma\underline{\sigma}}{|s|}\,\triple{( \mathbf u_{\lambda}, \theta_{\varsigma} )}^2_{|s|, \mathbb{R}^3\setminus \Gamma}, \label{eq:4.16}
 \end{align}
 where $c:= \min\{1, \gamma/ \eta \} $.
 This shows that the sesquilinear form in \eqref{eq:4.13} is coercive  (see, e.g., in \cite{HsSaSaWe:2016}). Consequently, we have the result:
\begin{corollary} The transmission problem (\ref{eq:4.10}) as well as (\ref{eq:4.9}) has a unique solution
in ${\bf H}^1(\mathbb{R}^3\setminus \Gamma) \times H^1(\mathbb{R}^3\setminus \Gamma)$.
\end{corollary}

Similarly, for the double-layer potential $  \mathcal{D} (s)$, we may consider the transmission problem:
\begin{proposition}\label{pro:4.3}
 Let $ ({\boldsymbol{\phi}}, \varphi) \in {\bf H}^{1/2}(\Gamma) \times H^{1/2}(\Gamma) $ and consider the function $({\bf u}_{\phi}, \theta_{\varphi}) :=  \mathcal{D} (s)(\boldsymbol{\phi}, 
 \varphi) \in {\bf H}^1(\mathbb{R}^3\setminus \Gamma) \times H^1(\mathbb{R}^3\setminus \Gamma)$. Then 
\begin{subequations} \label{eq:4.18}
\begin{align}
\mathbf{B}(\partial_x , s) (\mathbf u_{\phi}, {\theta}_{\varphi}) 
& = {\bf 0}\quad \mbox{in}\quad \mathbb{R}^3 \setminus \Gamma, \label{eq:18a}\\
\jump{({\bf u}_{\phi}, \theta_{\varphi})} & =(\boldsymbol{\phi}, \varphi) , \label{eq:18b}\\
\jump{\mathcal{R}_N ({ \bf u}_{\phi}, \theta_{\varphi})} & = {\bf 0}. \label{eq:18c}
\end{align}
\end{subequations}
Moreover, $({\bf u}_{\phi}, \theta_{\varphi}) \in {\bf H}^1(\mathbb{R}^3\setminus \Gamma) \times H^1(\mathbb{R}^3\setminus \Gamma) $ is a solution of {\em(}\ref{eq:4.18}{\em )} if and only if ${\mathbf u}_{\phi} \in {\bf H}^1(\mathbb{R}^3\setminus \Gamma)$ and $\theta_{\varphi} \in H^1(\mathbb{R}^3\setminus \Gamma)$ are such that $\jump{({\bf u}_{\phi}, \theta_{\varphi})}=(\boldsymbol{\phi}, \varphi)$ and
\begin{equation}
\mathcal{A}_{ \mathbb{R}^3\setminus \Gamma} \left( ( {\bf u}_{\phi}, \theta_{\varphi} ), ( {\bf v},  v) \right) =
\left\langle \mathcal{R}_N ( { \bf u}_{\phi}, \theta_{\varphi} ) |_{\Gamma}, \overline{\jump{({\bf v},  v)}}
\right\rangle_{\Gamma} = 0 \label{eq:4.20}
\end{equation}
for all $( {\bf v},  v) \in {\bf H}^1(\mathbb{R}^3\setminus \Gamma) \times H^1(\mathbb{R}^3\setminus \Gamma) $, where 
\[
\mathcal{R}_N ( { \bf u}_{\phi}, \theta_{\varphi} ) |_{\Gamma}=
\{ \mathcal{R}_N( { \bf u}_{\phi}, \theta_{\varphi} )\}^{\pm}.
\]
\end{proposition}
\begin{proof}
We extend $ ({\boldsymbol{\phi}}, \varphi) \in {\bf H}^{1/2}(\Gamma) \times H^{1/2}(\Gamma) $ to $\widetilde{({\boldsymbol{\phi}}, \varphi)} \in {\bf H}^1(\mathbb{R}^3\setminus \Gamma) \times H^1(\mathbb{R}^3\setminus \Gamma) $ such that $\jump{\widetilde{({\boldsymbol{\phi}}, \varphi)}} = (\boldsymbol{\phi}, \varphi) \in 
 {\bf H}^{1/2}(\Gamma) \times H^{1/2}(\Gamma)$ and let $ ({\bf w}, \varrho) : = ( {\bf u}, \theta ) - \widetilde{({\boldsymbol{\phi}}, \varphi)}$, where $ ( {\bf u}, \theta ) \in {\bf H}^1(\mathbb{R}^3\setminus \Gamma) \times H^1(\mathbb{R}^3\setminus \Gamma) $ is any pair satisfying \eqref{eq:4.18}. Then $ ({\bf w}, \varrho) \in {\bf H}^1(\mathbb{R}^3\setminus \Gamma) \times H^1(\mathbb{R}^3\setminus \Gamma) $ with $\jump{({\bf w}, \varrho)} = {\bf 0} $ and satisfies the variational equation:
\begin{align*} 
\mathcal{A}_{ \mathbb{R}^3\setminus \Gamma} \left( ({\bf w}, \varrho) , ( {\bf v},  v) \right) 
& = - \mathcal{A}_{ \mathbb{R}^3\setminus \Gamma} \Big( \widetilde{({\boldsymbol{\phi}}, \varphi)} , ( {\bf v},  v) \Big) 
\end{align*}
for all $( {\bf v},  v) \in {\bf H}^1(\mathbb{R}^3\setminus \Gamma) \times H^1(\mathbb{R}^3\setminus \Gamma)$. Thus, starting from \eqref{eq:4.16}
\begin{align*}
c~\frac{\sigma\underline{\sigma}}{|s|} \triple{ ({\bf w}, \varrho) }^2_{|s|, \mathbb{R}^3\setminus \Gamma}
&\leq \mathrm{Re} \left( \mathcal{A}_{ \mathbb{R}^3\setminus \Gamma} \left(Z(s)({\bf w}, \varrho) , \overline{({\bf w}, \varrho) } \right) \right)\\
&\leq \left|\mathcal{A}_{\mathbb{R}^3\setminus \Gamma}\left(Z(s) \widetilde{({\boldsymbol{\phi}}, \varphi)}, \overline{({\bf w}, \varrho) }\right)\right|\\
& \leq c\frac{|s|}{\underline{\sigma}} \triple{\widetilde{({\boldsymbol{\phi}}, \varphi)}}_{|s|, \mathbb{R}^3\setminus \Gamma}~\triple{({\bf w}, \varrho )}_{|s|, \mathbb{R}^3\setminus \Gamma}. 
\end{align*}
This yields the estimate 
\begin{equation*}
 \triple{({\bf w}, \varrho) }_{1, \mathbb{R}^3\setminus \Gamma}
 \leq c \frac{|s|^3} {\sigma\underline{\sigma}^{9/2}} ~ \triple{\widetilde{({\boldsymbol{\phi}}, \varphi)}}_{1, \mathbb{R}^3\setminus \Gamma}, 
\end{equation*}
and $ ( {\bf u}, \theta) = \widetilde{ ({\boldsymbol{\phi}}, \varphi)} + ({\bf w}, \varrho)  $ is the unique solution.
In particular, we see from \eqref{eq:4.20} that if $ \widetilde{({\boldsymbol{\phi}}, \varphi)} =  \mathcal{D} (s)(\boldsymbol{\phi}, \varphi)$, then $ ( {\bf w}, \varrho) \equiv {\bf 0}.$
\end{proof}

We notice that 
\begin{align}\label{eq:4.21}
\mathcal{A}_{\mathbb{R}^3\setminus \Gamma}\left(  Z(s)({\bf u}_{\phi}, \theta_{\varphi} ), \overline{ ( {\bf u_{\phi}}, \theta_{\varphi})} \right)&= \left\langle Z(s) \mathcal{R}_N ( { \bf u}_{\phi}, \theta_{\varphi} ) |_{\Gamma}, \overline{ \jump{({\bf u}_{\phi}, \theta_{\varphi}) } }\right\rangle_{\Gamma}.
\end{align}
Again making use of the inequality \eqref{eq:4.16}  for the solution $( {\bf u}_{\phi}, \theta_{\varphi} )$, we may establish the existence and uniqueness of results for the 
transmission problem (\ref{eq:4.18}):
\begin{corollary} There exists a unique solution solution of the transmission problem \eqref{eq:4.18} in 
$ {\bf H}^1(\mathbb{R}^3\setminus \Gamma) \times H^1(\mathbb{R}^3\setminus \Gamma) $. 
\end{corollary}
\begin{definition}The hypersingular boundary integral operator $\mathcal{W}(s)$ is a continuous linear mapping ${\bf H}^{1/2}(\Gamma) \times H^{1/2}(\Gamma) \rightarrow \ {\bf H}^{-1/2}(\Gamma) \times H^{-1/2}(\Gamma)$ defined by 
\[
\mathcal{W}(s)({\bm \phi}, \varphi) := - \Big\{ {\mathcal R}_{N} \mathcal{D}(s)({\bm \phi}, \varphi)\Big\}^{\mp}\quad \mbox{for} \quad ({\bm \phi}, \varphi) \in {\bf H}^{1/2}(\Gamma) \times H^{1/2}(\Gamma).
\]
\end{definition}

From the weak formulations of the transmission problems \eqref{eq:4.9} and \eqref{eq:4.18}, one may obtain 
coercivity of the corresponding boundary integral operators. Let us begin with the hypersingular boundary integral operator $\mathcal{W}$. We note that from (\ref{eq:4.21}),
 \begin{align}
\mathrm{Re} \left(\left\langle Z(s) \mathcal{W}(s)(\boldsymbol{\phi},\varphi), \overline{(\boldsymbol{\phi}, \varphi )} \right\rangle_{\Gamma} \right) &= \mathrm{Re} \left( \left\langle Z(s) \mathcal{R}_N ( { \bf u}_{\phi}, \theta_{\varphi} ) |_{\Gamma}, \overline{\jump{({\bf u}_{\phi}, \theta_{\varphi})} }\right\rangle_{\Gamma}\right) \nonumber\\
& = \mathrm{Re} \left( \mathcal{A}_{\mathbb{R}^3\setminus \Gamma} \left( Z(s) ({\bf u}_{\phi}, \theta_{\varphi} ), \overline{({\bf u}_{\phi}, \theta_{\varphi})}\right) \right) \nonumber\\
\nonumber
&\geq ~ c \,\frac{\sigma\underline{\sigma}}{|s|}\,\triple{( \mathbf u_{\phi}, \theta_{\varphi} )}^2_{|s|, \mathbb{R}^3\setminus \Gamma},\\
\label{eq:4.24}
&\geq  \frac{\sigma\underline{\sigma}^3}{|s|}\,\triple{( \mathbf u_{\phi}, \theta_{\varphi} )}^2_{|1|, \mathbb{R}^3\setminus \Gamma}  .  
\end{align}
Hence, from the trace theorem it follows that
\[
\mathrm{Re} \left(\left\langle Z(s) \mathcal{W}(s)(\boldsymbol{\phi},\varphi), \overline{(\boldsymbol{\phi}, \varphi )} \right\rangle_{\Gamma} \right) \geq C_\Gamma  \frac{\sigma\underline{\sigma}^3}{|s|}\| (\boldsymbol{\phi}, \varphi) \|^2_{{\bf H}^{1/2} (\Gamma) \times H^{1/2} (\Gamma)}.
\]
Now for the weakly singular boundary integral operator $\mathcal{V}(s)$, it follows from \eqref{eq:4.16} that
\[
\mathrm{Re} \left(\left\langle Z(s)(\boldsymbol{\lambda},\varsigma ), \overline{\mathcal{V}(s)(\boldsymbol{\lambda},\varsigma)} \right\rangle_{\Gamma} \right) 
 \geq c \, \frac{\sigma\underline{\sigma}}{|s|} \triple{ (\mathbf{u}_{\lambda}, \theta_{\varsigma})}^2_{|s|, \mathbb{R}^3\setminus \Gamma}.
\]
It remains to bound $\triple{(\mathbf u_{\lambda}, \theta_{\varsigma})}^2_{|s|, \mathbb{R}^3\setminus \Gamma}$ below by $ 
\| (\boldsymbol{\lambda},\varsigma )\|^2_{{\bf H}^{-1/2} (\Gamma) \times H^{-1/2} (\Gamma)} $. From the jump condition, we see that 
\begin{equation} \label{eq:4.28}
\|(\boldsymbol{\lambda}, \varsigma) \|_{{\bf H}^{-1/2} (\Gamma)}
 = \| \jump{\mathcal{R}_N ({\bf u}_{\lambda}, \theta_{\varsigma} )}\|_ {{\bf H}^{-1/2} (\Gamma)}.
\end{equation}
Since we do not have a lifting -type lemma, we need here a generalized first Green theorem (see, e.g. \cite{HsWe:2008}). By the definition, 
\[
 \|\{ \mathcal{R}_N ({\bf u}_{\lambda}, \theta_{\varsigma} )\}^{\pm} \|_ {{\bf H}^{-1/2} (\Gamma)}
 := \sup \frac{\langle \{\mathcal{R}_N({\bf u}_{\lambda}, \theta_{\varsigma})\}^{\pm}, 
 ({\boldsymbol \varphi},  \varphi) \rangle} { \|(\boldsymbol \varphi,  \varphi)\|_{{\bf H}^{+1/2} (\Gamma)} }.
 \]
First, we extend $ (\boldsymbol \varphi, \varphi) \in {\bf H}^{1/2} (\Gamma)$ to $(\tilde{\bf v}, \tilde{ v}) \in {\bf H}^{1}
(\Omega^{\mp})$ so that the traces $\{\tilde{\bf v}, \tilde{ v}\}^{\mp} = (\boldsymbol\varphi ,  \varphi) \in {\bf H}^{1/2} (\Gamma)$. Then from the first Green formula \eqref{eq:3.1}, we see that 
 \begin{align*}
\left\langle \{\mathcal{R}_N ({\bf u}_{\lambda}, ~ \theta_{\varsigma} ) \}^{\pm}, 
 (\boldsymbol \varphi,  \varphi) \right\rangle 
 &\leq  \frac{|s|}{\underline{\sigma}} \triple{{\bf u}_{\lambda}}_{|s|, \Omega^{\mp}} \triple{{\bf \tilde{v}}}_{ 1, \Omega^{\mp}} +\frac{ \sqrt{|s|}}{\sqrt{\underline{\sigma}}} \triple{\theta_{\varsigma} }_{|s|,\Omega^{\mp}} \triple{ \tilde{ v }}_{1, \Omega^{\mp}} \\
&\quad + \frac{\gamma }{\sqrt{|s|}} \triple{\theta_{\varsigma}}_{|s|,\Omega^{\mp}} \triple{ \tilde{\mathbf{v}}}_{1, \Omega^{\mp}} + |s| ~\eta~ \triple{{\bf u}_{\lambda} }_{|s|, \Omega^{\mp}} ~ \triple{\tilde{ v} }_{1, \Omega^{\mp}}\\
&\leq \left(\frac{|s|}{\underline{\sigma}} +\frac{\gamma}{\sqrt {|s|}}+ \eta |s| \right) \triple{({\bf u}_{\lambda}, \theta_{\varsigma})}_{|s|, \Omega^{\mp}}~ \triple{(\tilde{\bf v}, \tilde{ v})}_{1, \Omega^{\mp}}.
\end{align*}
We now estimate the coefficient $ \left(\frac{|s|}{\underline{\sigma}} +\frac{\gamma}{\sqrt{|s|}}+ \eta |s| \right) $ separately as follows:
\begin{align*}
 \left(\frac{|s|}{\underline{\sigma}} +\frac{\gamma}{\sqrt {|s|}}+ \eta |s| \right)
&\leq \max \{ 1, \gamma, \eta\} \left( \frac{|s|}{\underline{\sigma}} + \frac{1}{\sqrt{|s|}} + |s| \right) \\
&\leq \max \{ 1, \gamma, \eta\} \frac{|s|}{\underline{\sigma}} \left( 1 +  \frac{1}{ \sqrt {|s|} } + 1 \right) \\ 
&\leq \max \{ 1, \gamma, \eta\} \frac{|s|}{\underline{\sigma}} \left( 1 +  \frac{1}{ \sqrt {\underline{\sigma}} } + 1 \right) \leq c ~ \frac{|s|}{\underline{\sigma}^{3/2}}
\end{align*}
where $c:= 3 \max \{ 1, \gamma, \eta\} $. This yields
\begin{align}\label{eq:4.30}
 \|\mathcal{R}_N ({\bf u}_{\lambda}, \theta_{\varsigma} )\|_ {{\bf H}^{-1/2} (\Gamma)}
 &\leq c~ \frac{|s|} {\underline{\sigma}^{3/2} }\triple{({\bf u}_{\lambda}, \theta_{\varsigma})}_{|s|, \Omega^{\mp}},
 \end{align}
which in combination with \eqref{eq:4.28} implies that
\begin{equation}\label{eq:4.31}
 \|( \boldsymbol{\lambda}, \varsigma) \|^2_{ {\bf H}^{-1/2} (\Gamma) } \leq ~c~ \frac{|s|^2} {{\underline{\sigma}}^3} 
 \triple{( {\bf u}_{\lambda}, \theta_{\varsigma} )}^2_{|s|, \mathbb{R}^3\setminus \Gamma}.
\end{equation}
As consequence, the boundary operators are coercive and we have:
\begin{lemma} \label{le:4.5}
There hold the coercivity estimates
\begin{align*}
 \left|\left\langle (\boldsymbol{\lambda},\varsigma), \overline{\mathcal{V}(s)(\boldsymbol{\lambda}, \varsigma)} \right\rangle_{\Gamma}\right| 
&\geq c_V \frac{\sigma\underline{\sigma}^5} {|s|^4} 
\|( \boldsymbol{\lambda}, \varsigma) \|^2_{ {\bf H}^{-1/2} (\Gamma) } \quad \mbox{for all} \quad (\boldsymbol{\lambda},\varsigma) \in {\bf H}^{-1/2} (\Gamma), \\ 
\left|\left\langle \mathcal{W}(s)(\boldsymbol{\phi},\varphi), \overline{(\boldsymbol{\phi}, \varphi )} \right\rangle_{\Gamma}\right| 
& \geq c_W \frac{\sigma\underline{\sigma}^4}{|s|^2} \| (\boldsymbol{\phi}, \varphi) \|^2_{{\bf H}^{1/2} (\Gamma)}
\quad \mbox{for all} \quad (\boldsymbol{\phi}, \varphi) \in \mathbf{H}^{+1/2} (\Gamma),
\end{align*}
where $c_V $ and $ c_W$ are positive constants depending only on the geometry and the parameters $\gamma$ and $\eta$.
These coercivity estimates ensure the invertibility of $\mathcal{V}(s)$ and $\mathcal{W}(s)$ by a Lax-Milgram argument and provide a bound for the inverse of $\mathcal{V}(s)$ and $\mathcal{W}(s)$ such that 
 \begin{align}
\|\mathcal{V}^{-1}(s) \|_{ \mathbf{H}^{1/2}(\Gamma) \rightarrow \mathbf{H}^{-1/2}(\Gamma)} \leq c^{-1}_V 
 \frac {|s|^4} {\sigma\underline{\sigma}^5 }, \label{eq:4.34}\\
\|\mathcal{W}^{-1}(s) \|_{ \mathbf{H}^{-1/2}(\Gamma) \rightarrow \mathbf{H}^{+1/2}(\Gamma)} \leq 
 c^{-1}_W \frac{|s|^2} {\sigma\underline{\sigma}^4} . \label{eq:4.35}
 \end{align}
\end{lemma}

The next lemma concerns the bounds for the operators $\mathcal{S}(s), \mathcal{V}(s),  \mathcal{D}(s) $ and $\mathcal{W}(s)$.
\begin{lemma}\label{le:4.6} For all $s \in \mathbb{C}_{+}$ the following estimates hold 
\begin{align*}
(a)\quad \|  \mathcal{S} (s) \|_{ {\bf H}^{-1/2} (\Gamma) \rightarrow {\bf H}^1(\mathbb{R}^3\setminus \Gamma)} \leq c  \frac{|s|^2}{\sigma\underline{\sigma}^4}, & \qquad (b)\quad 
\|\mathcal{V}(s) \|_{ {\bf H}^{-1/2} (\Gamma) \rightarrow {\bf H}^{1/2}(\Gamma) } \leq c \frac{|s|^2}{\sigma\underline{\sigma}^4 }, \\
(c)\quad \|  \mathcal{D} (s) \|_{ {\bf H}^{1/2} (\Gamma) \rightarrow {\bf H}^1(\mathbb{R}^3\setminus \Gamma)} \leq c ~\frac{|s|^3}{\sigma\underline{\sigma}^{9/2}}, & \qquad (d)\quad 
\| \mathcal{W}(s) \|_{ {\bf H}^{1/2} (\Gamma) \rightarrow {\bf H}^{-1/2}(\Gamma) } \leq c \frac{|s|^4}{\sigma\underline{\sigma}^5 }.
\end{align*}
where $c$ represents a generic positive constant depending on $\Gamma$ and the parameters $\gamma$ and $\eta$. 
\end{lemma}
\begin{proof}
 (a) From \eqref{eq:4.16} we know that
\[
\frac{\sigma\underline{\sigma}}{|s|} \triple{(\mathbf{u}_{\lambda}, \theta_{\varsigma})}^2_{|s|, \mathbb{R}^3\setminus \Gamma}
 \leq \mathrm{Re} \left( \left\langle Z(s)(\boldsymbol{\lambda},\varsigma ), \overline{\mathcal{V}(s)(\boldsymbol{\lambda},\varsigma )} \right\rangle_{\Gamma} \right)
 \leq c\frac{|s|}{\underline{\sigma}} \|( \boldsymbol{\lambda}, \varsigma) \|_{ {\bf H}^{-1/2} (\Gamma) }
 \triple{( {\bf u}_{\lambda}, \theta_{\varsigma} ) }_{1, \mathbb{R}^3\setminus \Gamma}.
 \]
It then follows from \eqref{eq:4.6} and the definition of $ \mathcal{S}(s)$, that
\[ 
 \triple{\mathcal{S}(s)(\boldsymbol{\lambda}, \varsigma)}_{1, \mathbb{R}^3\setminus \Gamma}= \triple{(\mathbf{u}_{\lambda}, \theta_{\varsigma})}_{1, \mathbb{R}^3\setminus \Gamma}
\leq c~ \frac{|s|^2}{\sigma\underline{\sigma}^4} \|( \boldsymbol{\lambda}, \varsigma) \|_{ {\bf H}^{-1/2} (\Gamma) }.
\]
This yields the desired bound for $ \mathcal{S}(s)$:
\[
\|\mathcal{S}(s)\|_{\mathbf{H}^{-1/2} (\Gamma) \rightarrow {\mathbf H}^1(\mathbb{R}^3\setminus \Gamma)}  \leq c \frac{|s|^2}{\sigma\underline{\sigma}^4} .
 \]
The above inequality also implies the the estimate (b), since 
 \[
 \|\mathcal{V}(s) ( \boldsymbol{\lambda}, \varsigma) \|_{\mathbf{H}^{1/2}(\Gamma)} = \|  \mathcal{S}(s)( \boldsymbol{\lambda}, \varsigma) |_\Gamma\|_{\mathbf{H}^{1/2}(\Gamma)}\leq c \triple{\mathcal{S}(s)(\boldsymbol{\lambda}, \varsigma)}_{1, \mathbb{R}^3\setminus \Gamma}\leq c~ \frac{|s|^2}{\sigma\underline{\sigma}^4} \|( \boldsymbol{\lambda}, \varsigma) \|_{ {\bf H}^{-1/2} (\Gamma) }.
 \]
 Proofs of (c) and (d) follow in the same manner. From (\ref{eq:4.24}), we see that 
 \begin{align}
\frac{\sigma \underline{\sigma} }{|s|}\triple{( \mathbf u_{\phi}, \theta_{\varphi})}^2_{|s|, \mathbb{R}^3\setminus \Gamma}&\leq
\mathrm{Re} \left( \left\langle Z(s) \mathcal{W}(\boldsymbol{\phi},\varphi), \overline{(\boldsymbol{\phi}, \varphi )} \right\rangle_{\Gamma} \right) \nonumber \\
&\leq \frac{|s|}{\underline{\sigma}}\, \|\mathcal{W}(s)(\boldsymbol{\phi} ,\varphi)\|_{{\bf H}^{-1/2}(\Gamma)} ~ \|{(\boldsymbol{\phi}, \varphi )} \|_{{\bf H}^{1/2}(\Gamma)} \nonumber \\
&= \frac{|s|}{\underline{\sigma}}\,\, \| \mathcal{R}_N ({\bf u}_{\phi}, \theta_{\varphi} )\|_ {{\bf H}^{-1/2} (\Gamma)} \|{(\boldsymbol{\phi}, \varphi )} \|_{{\bf H}^{1/2}(\Gamma)} \nonumber \\
 &\leq \frac{|s|^2} {\underline{\sigma}^{5/2}} \triple{({\bf u}_{\phi}, \theta_{\varphi})}_{|s|, \mathbb{R}^3\setminus \Gamma} ~\|{(\boldsymbol{\phi}, \varphi )} \|_{{\bf H}^{1/2}(\Gamma)}, \label{eq:BoundForW}
 \end{align}
the last inequality follows from (\ref{eq:4.30}) with $({\bf u}_{\lambda}, \theta_{\varsigma} )$ replaced by $ (\mathbf u_{\bm \phi}, \theta_{\varphi}) $. Combining the result above with the equivalence relations \eqref{eq:4.6} yields
\[
 \triple{\mathcal{D}(s)(\boldsymbol{\phi}, \varphi )}_{1, \mathbb{R}^3\setminus \Gamma} = \triple{ ( \mathbf u_{\bm \phi}, \theta_{\varphi})}_{1, \mathbb{R}^3\setminus \Gamma}
 \leq c ~\frac{|s|^3}{\sigma\underline{\sigma}^{9/2}} \|{(\boldsymbol{\phi}, \varphi )}\|_{{\bf H}^{1/2}(\Gamma)},
\]
which leads to the desired bound in (c). In order to prove (d), notice that from the sequence of inequalities in \eqref{eq:BoundForW} it follows that:
\begin{align*}
\|\mathcal{W}(s)(\boldsymbol{\phi} ,\varphi)\|^2_{\mathbf{H}^{-1/2}(\Gamma)} 
&\leq c~ \frac{|s|^2} {\underline{\sigma}^3} \triple{({\bf u}_{\bm \phi}, \theta_{\varphi})}^2_{|s|, \mathbb{R}^3\setminus \Gamma}\\
&\leq c~ \frac{|s|^3} {\sigma \underline{\sigma}^4} \mathrm{Re} \left( \left\langle Z(s) \mathcal{W}(s)(\boldsymbol{\phi},\varphi), \overline{(\boldsymbol{\phi}, \varphi )} \right\rangle_{\Gamma} \right) \\
&\leq c~ \frac{|s|^4} {\sigma\underline{\sigma}^5} \|\mathcal{W}(s)(\boldsymbol{\phi} ,\varphi)\|_{{\bf H}^{-1/2}(\Gamma)} ~\|(\boldsymbol{\phi}, \varphi ) \|_{{\bf H}^{1/2}(\Gamma)}.
\end{align*}
This leads to the bound for $\mathcal{W}(s)$ in (d).
\end{proof}
For completeness we conclude the section with bounds for the boundary integral operators $\mathcal{K} (s) := \ave{\mathcal{R}_D \mathcal{D} (s)}$ and $\mathcal{K}^{\prime}(s):=\ave{\mathcal{R}_N  \mathcal{S} (s)}$ which are contained in the following lemma.
\begin{lemma} 
For all $s \in \mathbb{C}_{+}$, there hold the bounds for the operators:
\begin{subequations}\label{eq:BoundsForK}
\begin{align}
\label{eq:BoundsForKa}
\|\pm \tfrac{1}{2}\mathcal{ I} + \mathcal{K}^\prime(s) \|_{ \mathbf{ H}^{-1/2} (\Gamma) \rightarrow \mathbf{ H}^{-1/2} (\Gamma)}
& \leq c\, \frac{|s|^4}{\sigma\underline{\sigma}^{7}},\\
\label{eq:BoundsForKb} 
\|\mp \tfrac{1}{2}\mathcal{ I} + \mathcal{K}(s) \|_{ \mathbf{ H}^{1/2} (\Gamma)\rightarrow 
 \mathbf{ H}^{1/2} (\Gamma)}
& \leq c\, \frac{|s|^3} {\sigma\underline{\sigma}^{9/2}}.
\end{align}
\end{subequations}
\end{lemma}
\begin{proof}
Let $ ({\bm{\lambda}}, \varsigma) \in {\bf H}^{-1/2}(\Gamma) \times H^{-1/2}(\Gamma) $ and consider the simple-layer potential defined by $({\bf u}_{\lambda}, \theta_{\varsigma}) :=  \mathcal{S} (s)(\boldsymbol{\lambda},\varsigma) \quad \mbox{in} \quad \Omega^- ( \text{or} ~ \Omega^+)$. Then from \eqref{eq:4.30}, and the estimate for $\mathcal S(s)$ in Lemma \ref{le:4.6} (a) we obtain 
\begin{align*}
\|( \pm \tfrac{1}{2}\mathcal{ I} + \mathcal{K}^\prime(s))(\boldsymbol{\lambda}, \varsigma) \|_{ \mathbf{ H}^{-1/2} (\Gamma)}
& \leq c~ \frac{|s|} {\underline{\sigma}^{3/2}}\triple{\mathcal{S}(s)(\boldsymbol \lambda, \varsigma)}_{|s|, \Omega^{\mp}} \\
& \leq c~ \frac{|s|^2} {\underline{\sigma}^{3}} \triple{\mathcal{S}(s)(\boldsymbol \lambda, \varsigma)}_{1, \Omega^{\mp}}
\leq c~ \frac{|s|^4}{\sigma\underline{\sigma}^{7}} \|( \boldsymbol{\lambda}, \varsigma) \|_{ {\bf H}^{-1/2} (\Gamma) }.
\end{align*}
which shows \eqref{eq:BoundsForKa}.  Now let $ ({\bm{\phi}}, \varphi) \in {\bf H}^{1/2}(\Gamma) \times H^{1/2}(\Gamma) $ and consider the double-layer potential defined by $({\bf u}_{\phi}, \theta_{\varphi}) :=  \mathcal{D} (s)(\bm\phi, \varphi)$. Then from Lemma \ref{le:4.6} (c) 
\[
\|( \mp \tfrac{1}{2}\mathcal{ I} + \mathcal{K}(s))({\bm{\phi}}, \varphi) \|_{ \mathbf{ H}^{1/2} (\Gamma)}
\leq c\,\triple{\mathcal{D}(s)(\bm\phi, \varphi)}_{1, \Omega^{\mp}}
\leq c\,\frac{|s|^3}{\sigma\underline{\sigma}^{9/2}} \|(\bm{\phi}, \varphi )\|_{{\bf H}^{1/2}(\Gamma)},
\]
which leads to the bound \eqref{eq:BoundsForKb}.
\end{proof}
%
\section{Results in the time domain}\label{sec:TimeDomain}
%
We now return to the time domain. As basic model problems, we consider the initial -Dirichlet and the initial-Neumann boundary value problems for the elastic displacement field ${\bf U}(x,t)$ 
and temperature filed ${\Theta}(x,t)$ governed by the linear thermo-elasto-dynamic equations \eqref{eq:2.1} and \eqref{eq:2.2} in $Q_T := \{ (x,t) : x \in \Omega^+, t \in (0, T] \}$ satisfying the homogeneous initial conditions
\begin{gather}
 {\bf U}(x,t) = {\bf 0}, ~ {\bf U_t}(x,t )= {\bf 0}, \quad \mbox{and} \quad {\Theta}(x,t) = 0 
 \quad \mbox{for} \quad -\infty < t \leq 0, x \in \Omega^+. \nonumber
\end{gather}
together with either
\begin{description}
\item (a) Dirichlet boundary conditions: 
\[
{\bf U}(x,t) = {\bf F}(x,t), \quad \mbox{and} \quad \Theta(x,t) = F(x,t)\quad \mbox{ on } \quad \Gamma_T := \Gamma \times (0, T\,], 
\]
\item (b) Neumann boundary conditions:
\[
\boldsymbol{\sigma} ( {\bf U}, \Theta ) {\bf n} = {\bf G}(x,t) \quad\mbox{ and}\quad 
\frac{\partial}{\partial n} \Theta (x,t) = G(x,t) \quad \mbox{on}\quad \Gamma_T.
\]
\end{description}
where ${\bf F}(x,t)$ and $F(x,t)$ as well as ${\bf G}(x,t)$ and $G(x,t)$ are given smooth functions. 
Since these are exterior problems, it is natural to construct solutions of the problems by using boundary integral methods. This means we first need a fundamental solution for the time-dependent problem (\ref{eq:2.1}) and (\ref{eq:2.2}). 
However, the corresponding fundamental solution may not be available in general for time dependent equations, or may be considerably more difficult to construct and to handle computationally than those of time-independent equations. This leads us to consider the problems by using Lubich's approach \cite{Lu:1994} in the Laplace domain which does not require the availability of fundamental solutions for the time-dependent equations. 

As in Section \ref{sec:GoverningEquations}, let us denote
\[
\mathbf{u}:= \mathbf{u}(x,s)= \mathcal{L}\{{\bf U}(x,t)\}, \quad \text{and}\quad  \theta:=\theta(x,s)= \mathcal L\{\Theta(x,t)\}.
\]
Then the model problems in the Laplace domain read: Find the solution pair $(\mathbf{u}, \theta)$ satisfying the system of partial differential equations 
\begin{equation} \label{eq:5.4}
{\bf B}(\partial_x , s) \begin{pmatrix}
{\mathbf u} \\
{\theta}\\
\end{pmatrix}
:= \left (
\begin{array} {ll}
  \Delta^* - \rho s^2 & -\gamma \;\nabla \\
  -\eta~ s\; \nabla^{\top} & \Delta -s/\kappa \\
 \end{array} \right )
 \begin{pmatrix}
{\mathbf u} \\
\theta 
\end{pmatrix}
= {\bf 0} \quad \mbox{in} \quad \quad \Omega^+ , 
\end{equation}
together with either 
(a) Dirichlet boundary conditions: 
\begin{align} 
\mathcal R_D(\mathbf{u}, \theta) = ({\bf f}(x,s), f (x,s) )\quad \mbox{on} \quad \Gamma, \label{eq:5.5}
\end{align}
or (b) Neumann boundary conditions:
\begin{align}
\mathcal{R}_N(\mathbf{u}, \theta) = ( {\bf g}(x,s) , g(x,s) ) \quad \mbox{on} \quad \Gamma. \label{eq:5.6}
\end{align}
Here ${\bf f}(x,s) $ and $f(x,s) $ are, respectively, the Laplace transforms of the functions of $ {\bf F}(x,t)$ and $F(x,t)$, while ${\bf g}(x,s) $ and $g(x,s) $ are the corresponding Laplace transforms of $ {\bf G}(x,t)$ and $G(x,t)$. 

Let us begin with the Dirichlet boundary value problem; we are interested here in the weak solutions $(\mathbf{u}, \theta) \in {\bf H}^1( \Omega^+) \times H^1(\Omega^+)$ of \eqref{eq:5.4} and \eqref{eq:5.5}. In the last section, it is shown that for given 
$({\bf f} (x,s), f(x,s) ) \in ({\bf H}^{-1/2}(\Gamma) \times H^{-1/2}(\Gamma))$, the problem has a unique solution in the form 
\[
(\mathbf{u}, \theta) =  \mathcal{S} (s)(\boldsymbol{\lambda}, 
 \varsigma) \in {\bf H}^1(\Omega^+) \times H^1(\Omega^+),
 \]
 where $(\boldsymbol{\lambda}, \varsigma) \in {\bf H}^{-1/2}(\Gamma) \times H^{-1/2}(\Gamma)$
 is the unique solution of the boundary integral equation of the first kind
 \[
 \mathcal{V}(s) ( \boldsymbol{\lambda}, \varsigma ) = ({\bf f}(x,s), f(x,s) ) \quad \mbox{on}\quad \Gamma. 
\]
From \eqref{eq:5.6} and \eqref{eq:5.5}, we obtain the solutions in terms of the given data such that
\begin{align*}
( \boldsymbol{\lambda}, \varsigma ) &= \mathcal{V}^{-1}(s)( {\bf f}(x,s), f(x,s) ) , \\
(\mathbf{u}, \theta)& =  \mathcal{S}(s)\mathcal{V}^{-1}(s) ({\bf f}(x,s), f(x,s) ).
\end{align*} 
In the time domain, the counterparts to these identities can be expressed in terms of the convolution of functions
(or operators) 
\begin{align*}
\mathcal{L}^{-1} ( \boldsymbol{\lambda}, \varsigma ) & = \mathcal{L}^{-1}\{\mathcal{V}^{-1}(s)\} \ast 
({\bf F}(x,t), F(x,t)), \\
 ({\bf U}(x,t), {\Theta(x,t)} ) &= \mathcal{L}^{-1} \{  \mathcal{S}(s) \circ\mathcal{V}^{-1}(s) \} \ast 
({\bf F}(x,t), F(x,t)).
\end{align*}
Next for the Neumann problem, we again seek a solution in terms of a double- layer potential in the form:
 \[
(\mathbf{u}, \theta) =  \mathcal{D} (s)(\boldsymbol{\phi}, 
 \varphi) \in {\bf H}^1(\Omega^+) \times H^1(\Omega^+).
 \]
It has been shown in the last section that $( \boldsymbol{\phi}, \varphi) \in {\bf H}^{1/2}(\Gamma) \times H^{1/2}(\Gamma)$ is the unique solution of the boundary integral equation of the first kind
\[
 \mathcal{W}(s) ( \boldsymbol{\phi}, \varphi ) = ({\bf g}(x,s), g(x,s) ) \quad \mbox{on}\quad \Gamma. 
\]
In the same manner as for the Dirichlet problem, we can express the solutions $( \boldsymbol{\phi}, \varphi )$
 and $(\mathbf{u}, \theta) $ in terms of the given data with the help of convolution of functions in the time domain
\begin{align*}
\mathcal{L}^{-1} \{( \boldsymbol{\phi}, \varphi )\} & = \mathcal{L}^{-1}\{\mathcal{W}^{-1}(s)\} \ast 
(\mathbf{G}(x,t), G(x,t)) , \\
 (\mathbf {U}(x,t), \Theta(x,t) ) &= \mathcal{L}^{-1} \{  \mathcal{D}(s) \circ\mathcal{W}^{-1}(s)\} \ast 
(\mathbf{G}(x,t), G(x,t)).  
\end{align*}
To this end, it is worth mentioning that an essential feature of the above approach is that estimates of properties of solutions in the time domain can be obtained \textit{without applying directly the Laplace inverse transform} to solutions in the transformed domain. In fact, with the properties of operators and solutions that we have obtained in the transformed domains, we are now in a position to estimate their corresponding properties in the time domain \cite{Lu:1994}.  A crucial result pertaining an admissible class of symbols  will be employed for this purpose.  Before presenting the result, however, we will have to introduce some notation.

For Banach spaces $X$ and $Y$, let $\mathcal{B}(X, Y)$ denote the set of bounded linear operators from $X$ to $Y$. Then we will say that analytic function $A : \mathbb{C}_+ \rightarrow \mathcal{B}(X, Y)$ belongs to the class $ \mathcal{A} (\mu, \mathcal{B}(X, Y))$, if  there exists  $\mu \in \mathbb{R}$ such that 
\[
\|A(s)\|_{X,Y} \le C_A(\left(\mathrm{Re} (s)\right) |s|^{\mu} \quad \mbox{for}\quad s \in \mathbb{C}_+ ,
\]
where $C_A : (0, \infty) \rightarrow (0, \infty) $ is a non-increasing function such that 
\[
C_A(\sigma) \le \frac{ c}{\sigma^m} , \quad \forall \quad \sigma \in ( 0, 1]
\]
for some $m \ge 0$. Our careful efforts on Section \ref{sec:BoundaryValueProblems} on bounding the Laplace-domain boundary integral operators in terms of the Laplace parameter $s$ and its real part $\sigma$ have been conducive to showing that they belong to an appropriate class of symbols. The point being that for symbols satisfying these properties we can apply the following result that relates the Laplace domain existence, uniqueness and stability results to their counterparts for the time-domain problem. The following statement is an improved version of  \cite[Proposition 3.2.2]{Sayas:2016}, the proof of which can be found in \cite{Sayas:2016errata}.

\begin{proposition}[\cite{Sayas:2016errata}]\label{pr:5.1}
Let $A = \mathcal{L}\{a\} \in \mathcal{A} \in (m + \mu, \mathcal{B}(X, Y))$ with $\mu\in [0, 1)$ and $m$ be a non-negative integer.  If $ g \in \mathcal{C}^{m+1}(\mathbb{R}, X)$ is causal and its derivative $g^{(m+2)}$ is integrable, then $a* g \in \mathcal{C}(\mathbb{R}, Y)$ is causal and 
\[
\| (a*g)(t) \|_Y \le 2^{\mu} C_{\epsilon} (t) C_A (t^{-1}) \int_0^1 \|(\mathcal{P}_2g^{(m)})(\tau) \|_X \; d\tau,
\]
where 
\[C_{\epsilon} (t) := \frac{1}{2\sqrt{\pi}} \frac{\Gamma(\epsilon/2)}{\Gamma\left( (\epsilon+1)/2 \right) } \frac{t^{\epsilon}}{(1+ t)^{\epsilon}}, \qquad (\epsilon :=  1- \mu)
\]
and 
\[
(\mathcal{P}_2g) (t) =  g + 2\dot{g} + \ddot{g}.
\]
\end{proposition}
We now apply Proposition \ref{pr:5.1} to the solutions of the Dirichlet problem and collect the time-domain results in the following theorem:
\begin{theorem}[The Dirichlet problem]
Let $\mathcal{F} = ({\bf F}(x,t), F(x,t))$. \\
(a) For the densities $( \boldsymbol{\lambda}, \varsigma ) \in {\bf H}^{-1/2}(\Gamma) \times H^{-1/2}(\Gamma)$, 
If $ \mathcal{F} \in \mathcal{C}^5(\mathbb{R}, {\bf H}^{1/2} (\Gamma) ) $ is causal and $\mathcal{F}^{(6)} $ is integrable, then $\mathcal{L}^{-1}\{(\boldsymbol{\lambda}, \varsigma)\} \in \mathcal{C}( \mathbb{R}, {\bf H}^{-1/2}(\Gamma))$ is causal and 
\begin{align}\label{eq:5.17}
\|\mathcal{L}^{-1}\{( \boldsymbol{\lambda}, \varsigma ) \} \|_{{\bf H}^{-1/2}(\Gamma) } \leq c \frac{t} {(1+ t) } 
\,t\,\max \{1, t^5 \} \int_0^t \|\mathcal{P}_2 \mathcal{F}^{(4)}(\tau) \|_ {{\bf H}^{1/2} (\Gamma) } d \tau.
\end{align}
(b) For the solution $ (\mathbf{U}(x,t), \Theta(x,t) ) = \mathcal{L}^{-1} \{  \mathcal{S}(s) \circ\mathcal{V}^{-1}(s)\} \ast \mathcal{F}$. If $\mathcal{F} \in \mathcal{C}^4(\mathbb{R},
{\bf H}^{1/2}( \Gamma))$ is causal and $\mathcal{F}^{(5)}$ is integrable, then $ ({\bf U}(x,t), {\Theta(x,t)} )\in 
\mathcal{C}(\mathbb{R}, {\bf H}^1(\Omega^+))$ is causal and 
\begin{align}\label{eq:5.18}
\|({\bf U}(\cdot ,t), \Theta(\cdot,t) )\|_{{\bf H}^1(\Omega^+)} \leq c \frac{t}{(1 + t)}\,t\,\max \{1, t^{9/2} \} \int_0^t \|\mathcal{P}_2 \mathcal{F}^{(3)}(\tau) \|_ {{\bf H}^{1/2} (\Gamma) } d \tau. 
\end{align}
\end{theorem}
\begin{proof}
For (a), the relevant spaces are $X= {\bf H}^{1/2}(\Gamma)$, and  $Y= {\bf H}^{-1/2}(\Gamma)$. The key point of the proof of (a) is the estimate for  $\mathcal{V}^{-1}(s)$ given in \eqref{eq:4.34} in Lemma \ref{le:4.5}. From that bound and following the notation on Proposition \ref{pr:5.1} we have that $\mu = 0, m=4, \epsilon = 1$, and
 \[
 C_{\epsilon} = c\,\frac{t} {(1+t)}, \qquad C_A(t^{-1}) = c\, t\max \{1, t^5\},
 \]
which proves \eqref{eq:5.17}. 

For the case (b) we consider the spaces $X= {\bf H}^{1/2}(\Gamma)$, and  $Y= {\bf H}^{1}(\Omega^+)$. The proof of the estimates of the solution $({\bf U}(x,t), \Theta(x,t) ) \in {\bf H}^1( \Omega^+) \times H^1(\Omega^+) $ is more involved. We need to get a tighter bound for the operator $ \mathcal{S}(s)\mathcal{V}^{-1}(s)$,
than the one that would follow from the product of the bounds for $\mathcal{S}(s)$ and $\mathcal{V}^{-1}(s)$ proven earlier. To do so, we start from the estimate given by (\ref{eq:4.16})
 \begin{align*}
\frac{\sigma\underline{\sigma}}{|s|}\triple{ ( \mathbf u_{\lambda}, \theta_{\varsigma} )}^2_{|s|, \mathbb{R}^3\setminus \Gamma}&\leq \mathrm{Re}\left( \mathcal{A}_{\mathbb{R}^3\setminus \Gamma}\left(  Z(s) ( \mathbf u_{\lambda}, \theta_{\varsigma}), \overline{ ( \mathbf u_{\lambda}, \theta_{\varsigma})} \right)\right)\\
&= \mathrm{Re} \left( \left\langle Z(s) (\boldsymbol{\lambda},\varsigma), \overline{ {\mathcal{V}(s)} (\boldsymbol{\lambda},\varsigma)}\right \rangle_{\Gamma} \right) \\
&\leq \frac{|s|}{\underline{\sigma}}~\|(\boldsymbol{\lambda},\varsigma)\|_{{\bf H}^{-1/2}(\Gamma)}~ \|({\bf f}(x,s), f (x,s) ) \|_{{\bf H}^{1/2} (\Gamma)}\\
& \leq \frac{|s|^2} {\underline{\sigma}^{5/2} }~
\triple{( {\bf u}_{\lambda}, \theta_{\varsigma} )}_{|s|, \mathbb{R}^3\setminus \Gamma}~ \|({\bf f}(x,s), f (x,s) ) \|_{{\bf H}^{1/2} (\Gamma)}, \quad \mbox{due to (\ref{eq:4.31})}.
\end{align*}
 From this, and making use of the equivalence between the norms we obtain 
 \[
 \triple{( \mathbf u_{\lambda}, \theta_{\varsigma} )}_{1, \mathbb{R}^3 }
\leq c \frac{|s|^3} {\sigma\underline{\sigma}^{9/2}} \|( {\bf f}(x,s), f(x,s) ) \|_{{\bf H}^{1/2} (\Gamma)},
 \]
 which implies that 
 \begin{equation}
 \label{eq:5.19}
\| \mathcal{S}(s)\mathcal{V}^{-1}(s)\|_{{\bf H}^{1/2}(\Gamma) \rightarrow {\bf H}^1(\Omega^+)} \leq ~c \frac{|s|^3} { \sigma\underline{ \sigma}^{9/2}}.
\end{equation}
The estimate \eqref{eq:5.18} follows from Proposition \ref{pr:5.1}  by extracting the information $ \mu = 0, m = 3, \epsilon = 1$ and 
 \[
 C_{\epsilon} =\frac{t}{(1+t)}, \quad \text{ and } \quad C_{A}(t^{-1}) = t\, \max\{1, t^{9/2}\},
 \]
 from the inequality \eqref{eq:5.19}.
\end{proof}

\begin{theorem}[The Neumann problem] Let $\mathcal{G} = ({\bf G}(x,t), G(x,t))$. \\
(a) For the densities $( \boldsymbol{\phi}, \varphi) \in {\bf H}^{1/2}(\Gamma) \times H^{1/2}(\Gamma)$, 
if $ \mathcal{G} \in \mathcal{C}^2 ( \mathbb{R}, {\bf H}^{-1/2} (\Gamma) )$ is causal and $\mathcal{G}^{(3)} $ is integrable, then $\mathcal{L}^{-1}\{(\boldsymbol{\phi}, \varphi)\} \in \mathcal{C}( \mathbb{R}, {\bf H}^{1/2}(\Gamma))$ is causal and 
\begin{equation}\label{eq:5.20}
\|\mathcal{L}^{-1}\{( \boldsymbol{\phi}, \varphi ) \} \|_{{\bf H}^{1/2}(\Gamma) } \leq c \frac{t} {(1+ t) } 
 \,t\,\max \{1, t^4 \} \int_0^t \|\mathcal{P}_2 \mathcal{G}^{(2)}(\tau) \|_ {{\bf H}^{-1/2} (\Gamma) } d \tau,
\end{equation}
(b) For the solution $ ({\bf U}(x,t), \Theta(x,t) ) = \mathcal{L}^{-1} \{  \mathcal{D} \circ\mathcal{W}^{-1}(s) \} \ast ({\bf G}(x,t), G(x,t)) $, if $ \mathcal{G} \in \mathcal{C}^{3}( \mathbb{R}, {\bf H}^{-1/2} (\Gamma) )$ is causal and $\mathcal{G}^{(4)} $ is integrable, then $({\bf U}(x,t), \Theta(x,t) )\in \mathcal C(\mathbb R, \mathbf{H}^1(\Omega^+))$ is causal and we have the estimate: 
\begin{align}\label{eq:5.21}
\|({\bf U}(\cdot,t)\}, {\Theta(\cdot,t)} )\|_{{\bf H}^1(\Omega^+)} \leq \frac{t}{(1 + t)} \,t\,\max \{1, t^4 \} \int_0^t \|\mathcal{P}_2 \mathcal{G}^{(2)}(\tau) \|_ {{\bf H}^{1/2} (\Gamma) } d \tau. 
\end{align}
\end{theorem}
\begin{proof}
The result  \eqref{eq:5.20} for (a) follows immediately from the bound for $\mathcal{W}^{-1}(s)$  given by (\ref{eq:4.35}) in Lemma \ref{le:4.5}.

To prove (b), we need to derive a bound for the operator $\mathcal{D}(s) \mathcal{W}^{-1}(s)$. Once again, the product of the bounds for $\mathcal{W}^{-1}(s)$  and $\mathcal{D}(s)$, would yield an unnecessarily loose estimate. Instead, we observe that from \eqref{eq:4.24}, we have 
\begin{align*}
\frac{\sigma\underline{\sigma}}{|s|} \triple{( \mathbf u_{\phi}, \theta_{\varphi})}^2_{|s|, \mathbb{R}^3\setminus \Gamma}
&\leq \mathrm{Re} \left( \mathcal{A}_{\mathbb{R}^3\setminus \Gamma} \left( Z(s)({\bf u}_{\phi}, \theta_{\varphi} ), \overline{({\bf u}_{\phi}, \theta_{\varphi})}\right) \right)  \\
&= \mathrm{Re} \left( \left\langle Z(s) \mathcal{R}_N ( { \bf u}_{\phi}, \theta_{\varphi} ) |_{\Gamma}, \overline{\jump{ ({\bf u}_{\phi}, \theta_{\varphi}) }}\right\rangle_{\Gamma}\right). 
\end{align*}
This implies that 
\[
\frac{\sigma\underline{\sigma}}{|s|} \triple{( \mathbf u_{\phi}, \theta_{\varphi} )}_{|s|, \mathbb{R}^3\setminus \Gamma}
\leq \frac{|s|}{\underline{\sigma}^2}\,\|\mathcal{R}_N ( { \bf u}_{\phi}, \theta_{\varphi} )\|_{ {\bf H}^{-1/2}(\Gamma) }
= \frac{|s|}{\underline{\sigma}^2}\,\|( {\bf g}(x,s) , g(x,s) ) \|_{ {\bf H}^{-1/2}(\Gamma) }.
\]
Hence, using \eqref{eq:4.6} we have 
\[
 \triple{( \mathbf u_{\phi}, \theta_{\varphi} ) }_{1, \mathbb{R}^3\setminus \Gamma} \leq c \frac{|s|^2}{\sigma\underline{\sigma}^4}
 \|( {\bf g}(x,s) , g(x,s) ) \|_{ {\bf H}^{-1/2}(\Gamma) },
\]
and as a consequence
\begin{equation}\label{eq:5.22}
 \| \mathcal{D}(s) \mathcal{W}^{-1}(s) \|_{{\bf H}^{-1/2}(\Gamma) \rightarrow {\bf H}^1(\Omega^+)}
\leq c \frac{|s|^2}{\sigma\underline{\sigma}^4}. 
\end{equation}
The estimate \eqref{eq:5.21} for (b) follows from an application of Proposition \ref{pr:5.1} using the information from the estimate \eqref{eq:5.22} above.
\end{proof}

\paragraph{A personal note from the Authors}
\textit{This article was in a preliminary state of development at the time when Prof. Francisco-Javier Sayas received sudden bad news regarding his health. After the news were received, the work on the manuscript was paused. Francisco's untimely passing in April 2019 prevented the three of us from completing the work together. We are very sorry for the loss of a brilliant colleague at the peak of his career, an esteemed and dedicated mentor and a wonderful collaborator. Most of all, we miss our very dear friend.}\\

%
\section{Acknowledgements}
%

Tonatiuh S\'anchez-Vizuet was partially funded by the US Department of Energy. Grant No. DE-FG02-86ER53233.

\makeatletter
\newcommand\appendix@section[1]{%
\refstepcounter{section}%
\orig@section*{Appendix \@Alph\c@section: #1}%
\addcontentsline{toc}{section}{Appendix \@Alph\c@section: #1}%
}
\let\orig@section\section
\g@addto@macro\appendix{\let\section\appendix@section}
\makeatother

%
\begin{appendix}
%
\section{Fundamental solutions}
%
\subsection{3-D Fundamental solution}
%
In this appendix, we will construction the fundamental solution of eq.(\ref{eq:2.5}) by following H\"{o}rmander's \cite{Ho:1964} approach (see also Kupradze \cite{Ku:1979}, Banerjee \cite{Ba:1994}), which is essentially similar to the construction of the inverse of a matrix in linear algebra. Let $\mathbf{B}(\partial_x, s) $ be the matrix of the operators defined by equations \eqref{eq:2.3} as in (\ref{eq:2.5}):
\begin{equation} \label{eq:A.1}
\mathbf{B}(\partial_x , s) \begin{pmatrix}
{\mathbf u} \\
{\theta}\\
\end{pmatrix}
:= \left (
\begin{array} {ccc} 
  \Delta^* - \rho s^2 &\; &-\gamma \;\nabla \\
  -\eta~ s\; \nabla^{\top} & \; & \Delta -s/\kappa \\
 \end{array} \right )
 \begin{pmatrix}
{\mathbf u} \\
\theta 
\end{pmatrix}.
\end{equation}
Note that $\mathbf{B}(\partial_x, s) $ is a $4\times 4$ matrix of scalar operators. Fixing the fourth row and column (corresponding to the location in the matrix of the scalar operator $\Delta -s/\kappa$) and letting the indices $i,j$ run from one to three, the entries of $\mathbf{B}(\partial_x, s) $ are given as
\begin{eqnarray*}
B_{ij} &=& (\mu \Delta -  \rho s^2) \delta_{ij} + (\lambda + \mu) \frac{\partial^2}{\partial x_i \partial x_j}, \quad (i, j = 1, 2, 3)\\
B_{i4} &=& -\gamma \;\frac{\partial}{\partial_i}, \quad (i =1, 2, 3)\\
B_{4j} &=& -\eta \;s \;\frac{\partial}{\partial j}, \quad (j = 1, 2, 3) \\
B_{44} &=& \Delta - s/\kappa.
\end{eqnarray*}
We begin by computing the determinant of $ \mathbf{B}(\partial_x, s)$:
\begin{equation*}
det\; \mathbf{B} = \left | \begin{array}{ccc|c}
B_{11} & B_{12} & B_{13} & B_{14}\\
B_{21} & B_{22} & B_{23} & B_{24}\\ 
B_{31} & B_{32} & B_{33} & B_{34}\\ \hline
B_{41} & B_{42} & B_{43} & B_{44}
\end{array} \right |
 = \mu ^2 (\lambda + 2 \mu) \left(\Delta^4 + a_1 \Delta^3 + a_2 \Delta^2 + a_3 \Delta + a_4\right). 
\end{equation*}
Here the coefficients $a_1, a_2, a_3, $ and $a_4$ are given explicitly:
\begin{eqnarray*}
a_1 &= & - \frac{1}{ \mu^2 (\lambda + 2 \mu) } \left( \gamma\,\eta \,\mu^2\,s + \frac{s}{\kappa} \left(\mu^3 + (\lambda + \mu) \mu^2 \right) + 3 \mu^2 \rho s^2 + 2 \mu(\lambda + \mu) \rho s^2 \right),
\\
a_2 & = & \frac{1} {\mu^2 (\lambda + 2 \mu)} \left( 2 \gamma \,\eta \,\mu\,\rho\,s^3 + \frac{s}{\kappa} \left( 3 \mu^2 \rho s^2 + 2 \mu (\lambda + \mu) \rho s^2 \right) + 3 \mu \rho^2 s^4 + (\lambda + \mu) \rho^2 s^4 \right), 
\\
a_3 &= & - \frac{1}{\mu^2 (\lambda + 2 \mu)} \left(\gamma\, \eta \,\mu\, \rho^2\, s^5 + \rho^3 s^6 + \frac{s}{\kappa} \left( 3 \mu \rho^2 s^4 + (\lambda + \mu) \rho^2 s^4\right)\right),
\\
a_4 &= &\frac{s} {\kappa} (\rho^3 s^6 ).
\end{eqnarray*}
One may show that 
\begin{equation} \label{eq:A.2}
det \,\mathbf{B} =\mu^2 (\lambda + 2 \mu) (\Delta - \lambda_1^2) (\Delta - \lambda_2^2) (\Delta - \lambda_3^2)^2 , 
\end{equation}
provided the constants $\lambda_1^2, \lambda_2^2, \lambda_3^2$ satisfy the dispersion relations as in eq.(\ref{eq:2.7a}), namely, 
\begin{alignat}{6}
\lambda_1^2 + \lambda_2^2 &=\frac{s}{\kappa} + \frac{\gamma \,\eta\, s} {\lambda +2 \mu} + \lambda_p^2,  &\qquad \qquad& \lambda_p^2 = \frac{\rho \,s^2} {\lambda + 2 \mu}, \label{eq:A.3}\\
\lambda_1^2 \,\lambda_2^2 &= \frac{s}{\kappa}\, \lambda_p^2,  &\qquad \qquad& \lambda_3^2 = \frac{\rho \,s^2}{\mu}. \nonumber
\end{alignat}
Next, we need to compute the cofactor (i.e., minor with signs) $A_{ij}(\partial_x,s) $ of the elements $B_{ij}(\partial_x, s)$ in $\mathbf{B}(\partial_x, s)$. A simple but tedious computation yields for $i,j=1,2,3$ :
\begin{align*}
A_{ij} &= \left(\gamma\, \eta\, s -(\lambda + \mu) (\Delta - s/\kappa)\right) (\mu \Delta - \rho s^2) \frac{\partial^2}{\partial x_i \partial x_j}, \\
&\qquad + \delta_{ij} \left((\mu \Delta - \rho s^2)^2 (\Delta - s/\kappa) - \left(\gamma\, \eta\, s -(\lambda + \mu) (\Delta - s/\kappa)\right) (\mu \Delta - \rho s^2) \Delta\right) \\
&= \mu \left(\gamma\, \eta\, s - (\lambda + \mu) (\Delta - s/\kappa ) \right) (\Delta - \lambda^2_3) \frac{\partial^2}{\partial x_i \partial x_j} + \mu\, (\lambda + 2 \mu) \delta_{ij} (\Delta - \lambda^2_1) (\Delta - \lambda^2_2) (\Delta - \lambda_3^2) , \\
A_{i4}&= \eta \,s (\mu \Delta -\rho s^2)^2 \frac{\partial}{\partial x_i} = \eta \,s \mu^2 ( \Delta - \lambda^2_3)^2 \frac{\partial}{\partial x_i},  \\
A_{4j} &= \gamma(\mu \Delta - \rho s^2)^2 \frac{\partial}{\partial x_j} = \gamma\, \mu^2 ( \Delta - \lambda_3^2)^2 \frac{\partial}{\partial x_j}, \\
A_{44} &= ( \mu \Delta - \rho s^2)^2 \left( (\lambda + 2 \mu) \Delta - \rho s^2 \right)
= \mu^2 (\lambda + 2 \mu) (\Delta - \lambda_3^2)^2 (\Delta - \lambda_p^2).
 \end{align*}
 We note that the cofactor matrix $ A_{ij}(\partial_x, s)_{4 \times 4} $ contains the factor 
 $(\Delta - \lambda^2_3)$. We may rewrite it in the form 
 \begin{equation} \label{eq:A.5}
 A_{ij}(\partial_x, s) _{4 \times 4} = \widetilde{A}_{ij}(\partial_x, s)_{4 \times 4} (\Delta - \lambda^2_3).
 \end{equation}
Suppose now $\underline {\underline{ \bf E}}(x,y;s)$ is the fundamental matrix for $\mathbf{B}(\partial_x, s)$ such that 
\begin{equation}\label{eq:A.6}
\mathbf{B}(\partial_x, s) \underline {\underline{\bf E}}(x,y;s) = - \delta (x-y) \underline{\underline{\bf I}}.
\end{equation}
Replacing $\underline {\underline{ \bf E}}(x,y;s)$ in (\ref{eq:A.6}) by the adjoint of cofactor matrix 
$ A_{ij}(\partial_x, s)^{\prime}_{4 \times 4} \varphi, $ 
where $\varphi$ is a scalar function and the prime means the transposition, we see that
\begin{align*}
 \sum_{k=1}^4 B_{ik}(\partial_x, s) A_{jk}(\partial_x,s) \varphi 
 & \equiv det\, \mathbf{B}(\partial_x,s) \varphi \; \delta_{ij} \\
\hspace{2in} & =\left( \mu^2 (\lambda +2 \mu) (\Delta -\lambda^2_1) (\Delta - \lambda^2_2) (\Delta - \lambda_3^2)^2 \varphi\right)\; \delta_{ij} .
 \end{align*} 
 This shows that in order for $\underline {\underline{ \bf E}}(x,y;s)= A_{ij}(\partial_x, s)^{\prime}_{4 \times 4} \varphi,$ to be the matrix of the fundamental solution, we must require $\varphi$ to be the solution of 
\[ 
 \mu^2(\lambda +2 \mu) (\Delta -\lambda^2_1) (\Delta - \lambda^2_2) (\Delta - \lambda_3^2)^2 \varphi = - \delta (x-y).
\]
In view of \eqref{eq:A.5}, it suffices then to consider the equation
\[
 (\Delta -\lambda^2_1) (\Delta - \lambda^2_2) (\Delta - \lambda_3^2) \Phi = -\delta(x-y)
 \]
 for the function $\Phi =\mu^2 (\lambda +2 \mu) (\Delta - \lambda_3^2) \varphi$. A simple manipulation yields 
 \begin{equation}\label{eq:A.8a}
 \Phi = \sum_{k=1}^3 
\frac{1} {(\lambda^2_k - \lambda^2_{k+1})(\lambda^2_k - \lambda^2_{k+2})}\; 
\frac{e^{-\lambda_k |x - y|}} {4 \pi|x - y|} \quad \mbox{with} \quad \lambda_4 = \lambda_1,\; \lambda_5 = \lambda_2 
\end{equation}
and consequently we obtain 
\begin{equation}\label{eq:A.9}
\underline {\underline{\bf E}}(x,y;s) = \frac{1}{\mu^2 (\lambda + 2 \mu)} \widetilde{A}_{ij}(\partial_x, s)^{\prime} _{4 \times 4} \Phi,
\end{equation}
where we have
\begin{gather}
\frac{1}{\mu^2 (\lambda + 2 \mu)} \widetilde{A}_{ij}(\partial_x, s)^{\prime} _{4 \times 4} 
:= \left (
\begin{array} {c|c} 
\mu^{-1} (\Delta - \lambda^2_1) (\Delta - \lambda^2_2) \,\underline{\underline{\bf I}}\\ 
 + \mu^{-1} (\lambda + 2 \mu)^{-1}\;\left( \gamma\, \eta\, s - (\lambda + \mu)\right.\\
 \left.(\Delta - s/\kappa)\right)\; \nabla\nabla^{\top}& (\lambda + 2 \mu)^{-1}\eta (\Delta - \lambda^2_3) \;\nabla \\ [2mm] \hline\\
 \hspace{4mm} (\lambda + 2 \mu)^{-1}\;\eta~ s (\Delta - \lambda^2_3) \nabla^{\top} &\hspace{4mm} (\Delta -\lambda_p^2)(\Delta -\lambda^2_3) \\
 \end{array} \right ). 
\end{gather} 
Substitution of $\Phi$ from \eqref{eq:A.8a} into \eqref{eq:A.9} results finally in
\begin{equation}\label{eq:A.11}
\underline {\underline{\bf E}}(x,y;s) = \sum_{k=1}^3 \mathbf{D}_k(x, s) \frac{e^{-\lambda_k |x - y|}}{4 \pi |x - y |}, 
\end{equation}
where $\mathbf{D}_k(x, s)'s $ are matrices of differential operators given by
\begin{eqnarray} 
\mathbf{D}_1(x, s)&:=& \frac{1}{\rho s^2 (\lambda_1^2 - \lambda_2^2)}
 \left (
\begin{array} {l|l} 
\hspace{2mm}\;(\lambda^2_p - \lambda^2_2) \,\nabla \nabla^{\top}
&\hspace{4mm} \gamma \, \lambda^2_p \;\nabla \\ [2mm] \hline\\
\hspace{4mm} s\, \eta \,\lambda^2_p \,\nabla^{\top} & \hspace{4mm} \rho\, s^2\, (\lambda^2_1 - \lambda^2_p) 
\end{array} \right ), \\[4mm]
\mathbf{D}_2(x, s)&:=& \frac{1}{\rho s^2 (\lambda_2^2 - \lambda_1^2)}
 \left (
\begin{array} {l|l} 
\hspace{2mm}\;(\lambda^2_p - \lambda^2_1)\, \nabla \nabla^{\top}
& \hspace{4mm} \gamma \, \lambda^2_p \;\nabla \\ [2mm] \hline\\
\hspace{4mm} s\,\eta \,\lambda^2_p \,\nabla^{\top} &\hspace{4mm} \rho \,s^2\, (\lambda^2_2 - \lambda^2_p) 
\end{array} \right ), \\[1mm]
\mathbf{D}_3(x, s)&:=& \frac{1}{\rho s^2}
 \left (
\begin{array} {l|l} 
\hspace{2mm}\;\lambda^2_3 \;\underline{\underline{\bf I}} - \nabla \nabla^{\top}\hspace{2mm}
&\hspace{4mm} {\mathbf 0} \\ [2mm] \hline\\
{\hspace{8mm}\mathbf 0} &\hspace{4mm} {\mathbf 0}
\end{array} \right ). 
 \end{eqnarray}
This expansion for the fundamental solution in \eqref{eq:A.11} was motivated by the compact form presented in a celebrated paper by Cakoni \cite{Cakoni:2000} on the time-harmonic thermoelastic screen scattering problem in $\mathbb{R}^3$.
In the following, we will only include the derivation of the matrix of differential operators 
$\mathbf{D}_1(x, s)$, since derivations of the other two are similar.

In all the arguments below, the starting point is the substitution of \eqref{eq:A.8a} into \eqref{eq:A.9} and the relations \eqref{eq:A.3} will be used extensively. We begin with the terms on the final row and column of $\mathbf{D}_1(x, s)$, let $i,j \in\{1,2,3\}$, and treat the following four cases separately:

\begin{itemize}
\item  ${\displaystyle \mathbf{D}_1(x, s)_{i 4}:= \frac{1}{\rho s^2 (\lambda_1^2 - \lambda_2^2)} \Big( \gamma \, \lambda^2_p \;\nabla \Big). }$

Letting $j=4$ and $i = 1, 2, 3$ in \eqref{eq:A.9} yields
\begin{align*}
\left(\frac{(\lambda + 2 \mu)^{-1}\gamma (\Delta - \lambda^2_3) \;\nabla \,} {(\lambda^2_1 - \lambda^2_2)(\lambda^2_1- \lambda^2_3)} \right)\,\frac{e^{-\lambda_1 |x - y|}} {4 \pi|x - y|} & =
 \left( \frac{(\lambda + 2 \mu)^{-1}\gamma \;\nabla \,} {\lambda^2_1 - \lambda^2_2}\right) \frac{e^{-\lambda_1 |x - y|}} {4 \pi|x - y|} \\
& = \frac{1}{\rho s^2 (\lambda_1^2 - \lambda_2^2)} (\gamma\, \lambda^2_p\, \nabla)\, \frac{e^{-\lambda_1 |x - y|}} {4 \pi|x - y|} .
\end{align*}
\item ${\displaystyle \mathbf{D}_1(x, s)_{4j}:=  \frac{1}{\rho s^2 (\lambda_1^2 - \lambda_2^2)} \Big( s\, \eta\, \lambda^2_p \;\nabla ^{\top} \Big)}$.

Letting $i=4$  and  $j = 1, 2, 3$ in \eqref{eq:A.9}, it follows that
\begin{align*}
\left( \frac{(\lambda + 2 \mu)^{-1}\;\eta~ s (\Delta - \lambda^2_3) \nabla^{\top} } {(\lambda^2_1 - \lambda^2_2)(\lambda^2_1- \lambda^2_3)} \right) \frac{e^{-\lambda_1 |x - y|}} {4 \pi|x - y|} 
& = \left(\frac{(\lambda + 2 \mu)^{-1}\;\eta~ s \nabla^{\top}} {\lambda^2_1 - \lambda^2_2}\right) \frac{e^{-\lambda_1 |x - y|}} {4 \pi|x - y|} \\
& = \frac{1} { \rho \, s^2 (\lambda^2_1 - \lambda^2_2)} (s \, \eta\; \lambda^2_p\, \nabla^{\top} ) \frac{e^{-\lambda_1 |x - y|}} {4 \pi|x - y|}.
\end{align*}
\item ${\displaystyle \mathbf{D}_1(x, s)_{44} :  = \frac{1}{\rho s^2 (\lambda_1^2 - \lambda_2^2)} \Big( \rho \,s^2\, (\lambda^2_2 - \lambda^2_p) \Big) }$.

If both $i=j=4$ in \eqref{eq:A.9}, then
\begin{equation*}
\left(\frac{(\Delta - \lambda^2_p) (\Delta - \lambda^2_3)} {(\lambda^2_1 - \lambda^2_2) (\lambda^2_1- \lambda^2_3) } \right) \frac{e^{-\lambda_1 |x - y|}} {4 \pi|x - y|} = \frac{1} { \rho \, s^2 \, (\lambda^2_1 - \lambda^2_2)} \Big( \rho s^2\, (\lambda^2_1- \lambda^2_p) \Big)
\frac{e^{-\lambda_1 |x - y|} }{4 \pi|x - y|} .
\end{equation*}
\item ${\displaystyle  \mathbf{D}_1(x, s)_{ij}: = \frac{1}{\rho s^2 (\lambda_1^2 - \lambda_2^2)} \Big( (\lambda^2_p - \lambda^2_2) \,\nabla \nabla^{\top} \Big). }$

The block corresponding to $i,j\in\{1,2,3\}$ is more involved. Once again, starting from \eqref{eq:A.9}:
\begin{align}
\nonumber
& \frac{1} {(\lambda^2_1 - \lambda^2_2)(\lambda^2_1- \lambda^2_3)}\Big(\mu^{-1} (\Delta - \lambda^2_1) (\Delta - \lambda^2_2) \,\underline{\underline{\bf I}} \\
 \label{eq:A.15}
 & \qquad  + \mu^{-1} (\Delta - \lambda^2_1) 
 \mu^{-1} (\lambda + 2 \mu)^{-1}\;\Big[ \gamma\, \eta\, s - (\lambda + \mu)
(\Delta - s/\kappa)\Big]\; \nabla\nabla^{\top} \Big) \frac{e^{-\lambda_1 |x - y|}} {4 \pi|x - y|}. 
\end{align}
First, note that the contribution due to the term involving the identity vanishes, since
\begin{equation}\label{eq:fundamentalsolution}
\displaystyle
( \Delta - \lambda_1^2 )\frac{ e^{-\lambda_1|x -y|}}{ 4\pi|x -y|} = 0.
\end{equation}
The same is true for the corresponding term in $\mathbf D_2(x,s)$). Now let us consider the term in the square bracket :
\begin{align}
\nonumber 
\gamma\, \eta\, s & -  (\lambda + \mu) (\Delta - s/\kappa)  \\
\nonumber
& = \gamma\, \eta\, s - (\lambda + 2 \mu) (\Delta - s/\kappa) + \mu (\Delta - s/\kappa) \\
\nonumber
& = - (\lambda + 2 \mu)\left(- \frac{\gamma \eta\, s}{\lambda + 2 \mu} + \Delta - s/\kappa  - \lambda^2_p + \lambda^2_p \right) + \mu (\Delta - s/\kappa) \\
\nonumber
& = - (\lambda + 2 \mu)\left( \Delta - \left( s/\kappa + \frac{\gamma \eta\, s}{\lambda + 2 \mu} + \lambda^2_p \right) + \lambda^2_p \right) +\mu (\Delta - \lambda^2_1) + \mu ( \lambda^2_1- s/\kappa)\\
\nonumber
& = - (\lambda + 2 \mu)\left( \Delta -(\lambda^2_1 + \lambda^2_2 )+ \lambda^2_p \right)  + \mu (\Delta - \lambda^2_1) + \mu ( \lambda^2_1- s/\kappa) \\
\label{eq:A.16}
& = -(\lambda + 3 \mu)( \Delta - \lambda^2_1) - (\lambda + 2 \mu) ( \lambda^2_p - \lambda^2_2) + \mu (\lambda^2_1 - s/\kappa).
\end{align}
Due to \eqref{eq:fundamentalsolution}, the term $ -(\lambda + 3 \mu)( \Delta - \lambda^2_1)$ can be dropped out of the computation and the last two terms in  \eqref{eq:A.16} can be substituted instead of the term inside the square bracket in \eqref{eq:A.15}, yielding 
\begin{align}
\nonumber
 \frac{1} {(\lambda^2_1 - \lambda^2_2)(\lambda^2_1- \lambda^2_3)} & \left( 
 \mu^{-1} (\lambda + 2 \mu)^{-1}\; \Big[ \gamma\, \eta\, s - (\lambda + \mu)
(\Delta - s/\kappa)\Big]\; \nabla\nabla^{\top} \right)  \\
\nonumber
& = \frac{1} {(\lambda^2_1 - \lambda^2_2)(\lambda^2_1- \lambda^2_3)}\left( 
- \mu^{-1}( \lambda^2_p - \lambda^2_2) +\frac{\lambda^2_1 - s/\kappa}{\lambda + 2 \mu} \right) \nabla\nabla^{\top}  \\
\nonumber
& =\frac{1} { \rho s^2 (\lambda^2_1 - \lambda^2_2)(\lambda^2_1- \lambda^2_3)}\left( 
- \lambda^2_3\; ( \lambda^2_p - \lambda^2_2) +\lambda^2_p\, (\lambda^2_1 - s/\kappa)\right) \nabla\nabla^{\top}  \\
\nonumber
& =\frac{1} { \rho s^2 (\lambda^2_1 - \lambda^2_2) (\lambda^2_1- \lambda^2_3) } (\lambda^2_p - \lambda^2_2) (\lambda^2_1 - \lambda^2_3) 
 \nabla\nabla^{\top} \\
 \label{eq:SquareBracket}
& =\frac{1} { \rho s^2 (\lambda^2_1 - \lambda^2_2)} (\lambda^2_p - \lambda^2_2)  
 \nabla\nabla^{\top}.
\end{align}
The expression above  corresponds to the second term that arises by distributing the product on the left \eqref{eq:A.15} between the terms in the first and second lines. Therefore, distributing the product and substituting \ref{eq:SquareBracket} completes the derivation of $\mathbf{D}_1(x, s)$.

\end{itemize}

We note that for the adjoint equation to (\ref{eq:A.1}), if $\underline {\underline{\bf E}}^{*}(x, y;s)$ is the fundamental solution such that 
\begin{equation*}
\mathbf{B}^{*}(\partial_y, s) \underline {\underline{\bf E}}^{*}(x,y;s) = - \delta (y-x) \underline{\underline{\bf I}},
\end{equation*}
then we have 
\[
 \underline {\underline{\bf E}}^{*}(x,y;s) = \underline {\underline{\bf E}}^{\top}( x, y; s), 
\]
where $\underline {\underline{\bf E}}^{\top}( x, y ;s) $ is obtained from $\underline {\underline{\bf E}}(x, y; s)$
by transposing the rows and columns (see \cite{Ku:1979}, p.96, and \cite{HsWe:2008}, p.131). It is worth mentioning that we may recover the fundamental solutions given in \cite[p.95]{Ku:1979} (see also 
\cite{DaKo:1988} and \cite{Cakoni:2000}) for the sytstem of time-harmonic oscillation equations with $s$ and $\lambda^2_j $ replaced by $- i~\omega$ and $-\lambda^2_j$, respectively.
%
\subsection{2-D fundamental solution}
%
For interested readers, we have also included here the derivation of the fundamental solution 
 for the operator (\ref{eq:A.1}) in $\mathbb{R}^2$ as in the form of (\ref{eq:A.11}): 
\[
\underline {\underline{\bf E}}(x,y) = \sum_{k=1}^3 \mathbf{D}_k(x, s)\frac{1}{2\pi} K_0({\lambda_k |x - y|}), 
\]
where $K_0(\lambda_k|x-y|)$ is modified Bessel function of the first kind, and $\mathbf{D}_k(x, s)'s $ are matrices of differential operators given by
\begin{eqnarray} 
\mathbf{D}_1(x, s)&:=& \frac{1}{\rho s^2 (\lambda_1^2 - \lambda_2^2)}
 \left (
\begin{array} {l|l} 
\hspace{2mm}\;(\lambda^2_p - \lambda^2_2) \,\nabla \nabla^{\top}
&\hspace{4mm} \gamma \, \lambda^2_p \;\nabla \\ [2mm] \hline\\
\hspace{4mm} s\, \eta \,\lambda^2_p \,\nabla^{\top} & \hspace{4mm} \rho\, s^2\, (\lambda^2_1 - \lambda^2_p) 
\end{array} \right ), \label{eq:A.2} \\[4mm]
\mathbf{D}_2(x, s)&:=& \frac{1}{\rho s^2 (\lambda_2^2 - \lambda_1^2)}
 \left (
\begin{array} {l|l} 
\hspace{2mm}\;(\lambda^2_p - \lambda^2_1)\, \nabla \nabla^{\top}
& \hspace{4mm} \gamma \, \lambda^2_p \;\nabla \\ [2mm] \hline\\
\hspace{4mm} s\,\eta \,\lambda^2_p \,\nabla^{\top} &\hspace{4mm} \rho \,s^2\, (\lambda^2_2 - \lambda^2_p) 
\end{array} \right ), \\[4mm]
\mathbf{D}_3(x, s)&:=&- \frac{1}{\rho s^2}
 \left (
\begin{array} {l|l} 
\hspace{2mm}\;\nabla \nabla^{\top}- \lambda^2_3 \;\underline{\underline{\bf I}}\hspace{2mm}
&\hspace{4mm} {\mathbf 0} \\ [2mm] \hline\\
{\hspace{8mm}\mathbf 0} &\hspace{4mm} {\mathbf 0}
\end{array} \right ). 
 \end{eqnarray}
To derive (\ref{eq:A.16}), we begin with the operator $\mathbf{B}(\partial_x, s)$ in $\mathbb{R}^2$, namely
\[
\mathbf{B}(\partial_x , s) \begin{pmatrix}
{\mathbf u} \\
{\theta}\\
\end{pmatrix}
:= \left (
\begin{array} {ccc} 
  \Delta^* - \rho s^2 & \; & -\gamma \;\nabla \\ 
  -\eta~ s\; \nabla^{\top} & \; & \Delta - s/\kappa \\
 \end{array} \right )
 \begin{pmatrix}
{\mathbf u} \\
\theta 
\end{pmatrix},
\]
where the entries of the matrix are given as
\begin{alignat*}{6}
B_{ij} =\,& (\mu \Delta -  \rho s^2) \delta_{ij} + (\lambda + \mu) \frac{\partial^2}{\partial x_i \partial x_j}, &\qquad& (i, j = 1, 2),\\
B_{i3} =\,& -\gamma \;\frac{\partial}{\partial_i} &\qquad& (i =1, 2),\\
B_{3j} =\,& -\eta \;s \;\frac{\partial}{\partial j} &\qquad& (j = 1, 2), \\
B_{33} =\,& \Delta -\frac{s}{\kappa}\,. &&
\end{alignat*}
We begin with the determinant of $\mathbf{B}(\partial_x, s) $:
\begin{equation*}
det\; \mathbf{B} = \left | \begin{array}{ll|l}
B_{11} & B_{12} & B_{13} \\
B_{21} & B_{22} & B_{23} \\ \hline
B_{31} & B_{32} & B_{33} 
\end{array} \right |
 = \mu (\lambda + 2 \mu) \left( \Delta^3 + a_1 \Delta^2 + a_2 \Delta + a_3 \right). 
\end{equation*}
Here the coefficients $a_1, a_2$ and $a_3, $ are given explicitly :
\begin{eqnarray*}
a_1 &= & - \frac{1}{ \mu (\lambda + 2 \mu) } \left(\frac{s}{\kappa} \left( 2\mu^2 +( \gamma\, \eta \,\kappa + \lambda) \mu \right) + (\lambda +3 \mu) \rho \,s^2 \right),
\\
a_2 & = & \frac{1} {\mu (\lambda + 2 \mu)} \left( \frac{s}{\kappa} \left( \gamma\,\eta\, \kappa +\lambda + 3\mu\right)\, \rho s^2 + \rho^2\, s^4 \right), 
\\
a_3 &= & - \frac{1}{\mu (\lambda + 2\mu)} \left( \frac{s}{\kappa} \left( \rho^2 s^4\right) \right).
\end{eqnarray*}
One may show that 
\begin{eqnarray} \label{eq:A.8}
det \,\mathbf{B} &=&\mu(\lambda + 2 \mu) (\Delta - \lambda_1^2) (\Delta - \lambda_2^2) (\Delta - \lambda_3^2), 
\end{eqnarray}
provided the constants $\lambda_1^2, \lambda_2^2, \lambda_3^2$ satisfy the dispersion relations in \eqref{eq:2.7a} 
\begin{alignat*}{6}
\lambda_1^2 + \lambda_2^2 =\,& \frac{s}{\kappa} + \frac{\gamma \,\eta\, s} {\lambda +2 \mu} + \lambda_p^2, &\quad \qquad& \lambda_p^2 =\,& \frac{\rho \,s^2} {\lambda + 2 \mu}, \\
\lambda_1^2 \,\lambda_2^2 =\,& \frac{s}{\kappa}\, \lambda_p^2, &\quad \qquad& \lambda_3^2 =\,& \frac{\rho \,s^2}{\mu}.
\end{alignat*}
The cofactor of the element in $\mathbf{B}(\partial_x,s)$ for $i,j = 1, 2,$ is given by  
\begin{align*}
A_{ij} =\,& -\left( (\lambda + \mu) (\Delta - s/\kappa)-\gamma\, \eta\, s \right)  \frac{\partial^2}{\partial x_i \partial x_j}\\
&\,+ \delta_{ij} \left( (\mu \Delta - \rho s^2) (\Delta - s/\kappa) + \left( (\lambda + \mu)( \Delta - s/\kappa) - \gamma\, \eta\, s\right) \Delta \right)\\
=\,&-\left((\lambda + \mu) (\Delta - s/\kappa)-\gamma\, \eta\, s \right) \frac{\partial^2}{\partial x_i \partial x_j} \\
&\,+\left( (\lambda + 2 \mu) (\Delta - \lambda_1^2) (\Delta - \lambda_2^2)\right) \,\delta_{ij}, \\
A_{i3}=\,& \eta \,s (\mu \Delta -\rho s^2)\frac{\partial}{\partial x_i} = \eta \mu\,s ( \Delta - \lambda^2_3) \frac{\partial}{\partial x_i}, \\
A_{3j} =\,& \gamma(\mu \Delta - \rho s^2) \frac{\partial}{\partial x_j} = \gamma\, \mu( \Delta - \lambda_3^2) \frac{\partial}{\partial x_j},\\
A_{33} =\,& ( \mu \Delta - \rho s^2)^2 + (\lambda + \mu)( \mu \Delta - \rho s^2) \Delta 
= \mu(\lambda + 2 \mu) (\Delta - \lambda_3^2)(\Delta - \lambda_p^2).
 \end{align*}
Suppose now $\underline {\underline{ \bf E}}(x,y; s)$ is the fundamental matrix for $\mathbf{B}(\partial_x, s)$ such that 
\begin{equation}\label{eq:A.28}
\mathbf{B}(\partial_x, s) \underline {\underline{\bf E}}(x,y) = - \delta (x-y) \underline{\underline{\bf I}}.
\end{equation}
Replacing $\underline {\underline{ \bf E}}(x,y;s)$ in (\ref{eq:A.28}) by the adjoint of cofactor matrix 
$A_{ij}(\partial_x, s)^{\prime}_{3 \times 3} \varphi, $ 
where $\varphi$ is a scalar function and the prime means the transposition, we see that
\begin{equation*}
 \sum_{k=1}^3 B_{ik}(\partial_x, s) A_{jk}(\partial_x,s) \varphi 
 \equiv \delta_{ij} det\, \mathbf{B}(\partial_x,s) \varphi  = \delta_{ij} \mu (\lambda +2 \mu) (\Delta -\lambda^2_1) (\Delta - \lambda^2_2) (\Delta - \lambda_3^2) \varphi.
 \end{equation*} 
This shows that in order for $\underline {\underline{ \bf E}}(x,y;s)= A_{ij}(\partial_x, s)^{\prime}_{3 \times 3} \varphi,$ to be the matrix of the fundamental solution, we require $\varphi$ to be the solution of 
 \begin{eqnarray*} 
 \mu (\lambda +2 \mu)(\Delta -\lambda^2_1) (\Delta - \lambda^2_2) (\Delta - \lambda_3^2) \varphi = - \delta (x-y). 
 \end{eqnarray*}
A simple manipulation yields 
 \begin{equation}\label{eq:A.29}
 \varphi = \frac{1}{\mu (\lambda + 2 \mu)} \sum_{k=1}^3 
\frac{1} {(\lambda^2_k - \lambda^2_{k+1})(\lambda^2_k - \lambda^2_{k+2})}\; 
K_0(\lambda_k|x-y|) \quad \mbox{with} \quad \lambda_4 = \lambda_1,\; \lambda_5 = \lambda_2 
\end{equation}
and we have 
\[
 {A}_{ij}(\partial_x, s)^{\prime} _{3 \times 3} 
:= \left (
\begin{array} {l|l} 
\hspace{4mm}( \lambda + 2\mu) ^{-1} ( \Delta - \lambda_1^2)(\Delta - \lambda^2_2) \bf I\\ 
 - \left( (\lambda + \mu)(\Delta - s/\kappa) - \gamma\, \eta\, s\right) \nabla \nabla^{\top}
& \hspace{8mm} \mu \,\eta (\Delta - \lambda^2_3) \;\nabla \\ [3mm]\hline \\
\hspace{8mm}\mu\, \eta \, s (\Delta - \lambda^2_3) \nabla^{\top} & \hspace{4mm}
\mu (\lambda + 2 \mu) (\Delta -\lambda_p^2)(\Delta -\lambda^2_3) \\
 \end{array} \right ) .
\]
Substitution of $\varphi$ from (\ref{eq:A.29}) results finally
\[
\underline {\underline{\bf E}}(x,y;s) = \sum_{k=1}^3 \mathbf{D}_k(x, s) \frac{1}{2 \pi} K_0(\lambda_k |x - y |), 
\]
where $\mathbf{D}_k(x, s)'s $ are matrices of differential operators given by
\begin{eqnarray} 
\mathbf{D}_1(x, s)&:=& \frac{1}{\rho s^2 (\lambda_1^2 - \lambda_2^2)}
 \left (
\begin{array} {l|l} 
\hspace{2mm}\;(\lambda^2_p - \lambda^2_2) \,\nabla \nabla^{\top}
&\hspace{4mm} \gamma\,\lambda_p^2 \;\nabla \\ [2mm] \hline\\
\hspace{4mm} s\, \eta \,\lambda^2_p \,\nabla^{\top} & \hspace{4mm} \rho\, s^2\, (\lambda^2_1 - \lambda^2_p) 
\end{array} \right ), \\[4mm]
\mathbf{D}_2(x, s)&:=& \frac{1}{\rho s^2 (\lambda_2^2 - \lambda_1^2)}
 \left (
\begin{array} {l|l} 
\hspace{2mm}\;(\lambda^2_p - \lambda^2_1)\, \nabla \nabla^{\top}
& \hspace{4mm} \gamma \, \lambda^2_p \;\nabla \\ [2mm] \hline\\
\hspace{4mm} s\,\eta \,\lambda^2_p \,\nabla^{\top} &\hspace{4mm} \rho \,s^2\, (\lambda^2_2 - \lambda^2_p) 
\end{array} \right ), \\[4mm]
\mathbf{D}_3(x, s)&:=& \frac{1}{\rho s^2}
 \left (
\begin{array} {l|l} 
\hspace{2mm}\; \lambda^2_3 \;\underline{\underline{\bf I}} - \nabla \nabla^{\top} \hspace{2mm}
&\hspace{4mm} {\mathbf 0} \\ [2mm] \hline\\
{\hspace{8mm}\mathbf 0} &\hspace{4mm} {\mathbf 0}
\end{array} \right ) .
 \end{eqnarray}
 \end{appendix}
 
%

\end{document}